\documentclass[a4paper]{amsart}
\usepackage{amsmath,amssymb,amsthm}
\usepackage{amsfonts}

\usepackage{xcolor}
\usepackage{tikz}
\usepackage{multicol}
\usepackage{url}
\usepackage{cancel}
\usepackage{hyperref}

\definecolor{green3}{rgb}{0,0.6,0}

\newcommand{\R}{\mathbb R}
\newcommand{\C}{\mathbb{C}}
\newcommand{\Z}{\mathbb{Z}}
\newcommand{\N}{\mathbb{N}}
\newcommand{\Q}{\mathbb{Q}}
\newcommand{\K}{\mathbb{K}}

\newcommand{\T}{\mathcal{T}}
\newcommand{\cA}{\mathcal{A}}

\DeclareMathOperator\val{val}

\newcommand{\M}{\mathbf{M}}

\newtheorem{theorem}{Theorem}[section]
\newtheorem{lemma}[theorem]{Lemma}
\newtheorem{proposition}[theorem]{Proposition}
\newtheorem{corollary}[theorem]{Corollary}
\newtheorem{comp}{Computation}[section]
\newtheorem{notation}[theorem]{Notation}

\theoremstyle{definition}

\newtheorem{definition}[theorem]{Definition}
\newtheorem{example}[theorem]{Example}

\newtheorem{remark}[theorem]{Remark}

\title[Tropical Severi varieties
of univariate polynomials]{Arithmetics and combinatorics of
tropical Severi varieties of univariate polynomials}
\author[A. Dickenstein, M.I. Herrero, L.F. Tabera]{Alicia Dickenstein,
Mar\'ia Isabel Herrero, Luis Felipe Tabera}
\date{}

\address{A. Dickenstein and M.I. Herrero: Dto. de Matem\'atica, FCEN -
Universidad de Buenos Aires, and IMAS (UBA-CONICET),  Ciudad
Universitaria - Pab. I - C1428EGA Buenos Aires, Argentina}

\address{L.F. Tabera: Dto. Matem\'aticas,
Estad\'istica y Computaci\'on, Universidad de Cantabria, Spain}
\email{alidick@dm.uba.ar, iherrero@dm.uba.ar, taberalf@unican.es}

\begin{document}

\begin{abstract}
We give a description of the tropical Severi variety of
univariate polynomials of degree $n$ having two double
roots.  We show that, as a set,
it is given as the union of three explicit types of cones of maximal dimension $n-1$,
where only cones of two of these types are cones of the secondary fan of $\{0,
\dots, n\}$.
Through Kapranov's theorem, this goal is achieved by
a careful study of the possible valuations of the elementary
symmetric functions of the roots of a polynomial with two
double roots. Despite its apparent simplicity, the computation of the
tropical Severi variety has both combinatorial and arithmetic ingredients.

\end{abstract}

\maketitle

\section{Introduction}\label{sec:introduction}

Moduli spaces of various objects are of primary interest in algebraic geometry.
Among them, Severi varieties are classical objects which
give parameter spaces for nodal hypersurfaces.  Starting with~\cite{Mikh05},
techniques from tropical geometry have been interestingly applied to the study of
classical enumerative problems. Mikhalkin's correspondence theorem 
allows to compute tropically the number of planar curves of degree $d$ and any
number $\delta$ of nodes passing through $\frac{(d+3)d}{2}-\delta$
points in general position, that is, the degree of the Severi
variety parametrizing those curves or, more generally, the degree
of the Severi varieties of nodal curves with $\delta$ nodes
defined by polynomials with support in (the lattice points of) a
given lattice polygon.  For curves of degree $d \ge \delta+2$,
the degree of the Severi variety coincides with the Gromov-Witten invariant
$N_{d,\binom{d-1}{2}-\delta}$. We refer the reader to~\cite{B12} for an introduction
to tropical geometry techniques for algebraic curve counting problems.

Consider a (finite) lattice configuration $A$ in any dimension (for instance,
all the lattice points in the $d$-th dilate of a unit simplex).
Order the points in $A$ and
denote by $\mathcal{A}$ the matrix having these points as columns.
The secondary fan of $A$ (defined and studied in \cite{GKZ})
parametrizes the regular polyhedral subdivisions of
the configuration. In the case of a
single node, the Severi variety corresponding to polynomials with
support in $A$ is defined by the $A$-discriminant.
It was proven in \cite{DFS07} that the tropical $A$-discriminant
coincides with the Minkowski sum of the row span of $\mathcal{A}$
and the tropicalization of the kernel of $\mathcal{A}$, and it is
the union of some cones in the secondary fan of $A$. Severi varieties
of polynomials with support in $A$  and increasing number
of nodes form a natural stratification of the $A$-discriminant.

As a consequence of the general position of
the points, the tropical curves appearing in Mikhalkin's
correspondence theorem can be described by the associated regular
subdivision of the support. That is, the set of tropical curves
with a specified combinatorial type counted in Mikhalkin's formula
correspond to polyhedral cones in the associated secondary fan.
However, these cones are a fraction of all possible cones in the
associated tropical Severi variety. E. Katz noted in \cite{Katz}
that there are maximal cones that are not supported in cones of
the secondary fan of $A$. Thus, the combinatorial description of
the curves is not enough in many cases to decide if a tropical
curve given by a tropical polynomial is in the  corresponding
Severi variety or not. This behavior was also observed by J.~J.
Yang, who gave a partial description of the tropicalization of the
Severi varieties in \cite{Y1,Y2}. We explore this phenomenon and
give a full characterization in the univariate setting for the
case of $\delta=2$ nodes. Besides the combinatorial constraints,
we describe arithmetic restrictions.

A recent interesting paper of Esterov~\cite{Esterov15} defines
affine characteristic classes of affine varieties, which govern
equivariant enumerative problems. The computation of the
fundamental class of a hypersurface of the torus in the
corresponding affine cohomology ring, amounts to the computation
of the Newton polytope of a defining equation. In general, for a
subvariety of the torus of any codimension, the computation of its
fundamental class amounts to the computation of its tropical fan.
The codimension two strata of the $A$-discriminant are given by
the closure of the hypersurfaces with support in $A$ with one
triple root or those hypersurfaces with two nodes, that is, the
$2$-Severi variety.  Esterov gives affine Pl\"ucker formulas
relating these classes with the class of the $A$-discriminant.

Our setting is the following. We fix a natural number $n \ge 4$.
The Severi variety $Sev_n^2$ is the Zariski closure
of the set of univariate polynomials of degree $n$
having $2$ distinct double roots (over an algebraically closed field of characteristic zero).
$Sev_n^2$ is a rational variety (with a rational parametrization defined over $\Q$)
and thus irreducible, of affine dimension $n-1$. The
tropical Severi variety $\mathcal{T}(Sev_n^2)$ is the
tropicalization of $Sev_n^2$. Despite its apparent simplicity, the computation of the
tropicalization $\mathcal{T}(Sev_n^2)$ has both combinatorial and arithmetic ingredients.

We describe in Section~\ref{sec:cones} three types of cones with maximal dimension $n-1=n+1-2$,
termed I, II and III. Cones of
type I correspond to two
double roots with different valuations and arise from the transversal intersection of
two maximal dimensional cones in the discriminant. Cones of types II and III
correspond to roots which are not tropically generic: both roots need to have the same valuation,
and the relative interior of these cones correspond to the
tropicalization of polynomials with two nodes whose initial coefficients satisfy explicit algebraic relations. Moreover,
cones of type III intersect the interior of a maximal cone in the
secondary fan of our configuration $\{0, 1, \dots, n\}$ of
exponents. The existence of a second node imposes a ``hidden tie''
with arithmetic constraints.

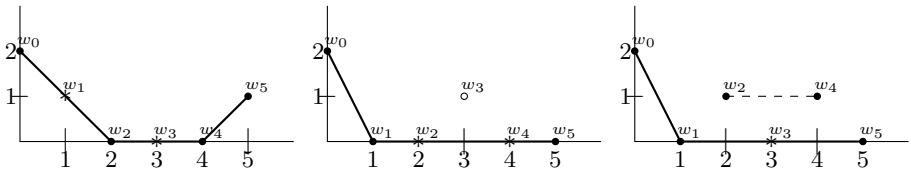
\begin{figure}[ht]
\begin{tabular}{lcr}
\begin{tikzpicture}[scale=0.6]
\draw (0,3) -- (0,0) -- (6,0); \draw (0,1) node {$-$} (1,0) node
{$|$} (5,0) node {$|$}; \draw (-0.2,1) node {\small{1}} (-0.2,2)
node {\small{2}}; \draw (1,-0.4) node {\small{1}} (2,-0.4) node
{\small{2}} (3,-0.4) node {\small{3}}; \draw (4,-0.4) node
{\small{4}} (5,-0.4) node {\small{5}}; \draw[fill] (0,2) circle
(2pt) (2,0) circle (2pt) (4,0) circle (2pt) (5,1) circle (2pt);
\draw[thick] (0,2) -- (2,0) -- (4,0) -- (5,1); \draw (1,1) node
{$*$} (3,0) node {$*$}; \draw (0.2,2.2) node {\tiny{$w_0$}}
(1.2,1.2) node {\tiny{$w_1$}} (2.2,0.2) node {\tiny{$w_2$}}; \draw
(3.2,0.2) node {\tiny{$w_3$}} (4.2,0.2) node {\tiny{$w_4$}}
(5.2,1.2) node {\tiny{$w_5$}};
\end{tikzpicture}
\begin{tikzpicture}[scale=0.6]
\draw (0,3) -- (0,0) -- (6,0); \draw (0,1) node {$-$} (3,0) node
{$|$}; \draw (-0.2,1) node {\small{1}} (-0.2,2) node {\small{2}};
\draw (1,-0.4) node {\small{1}} (2,-0.4) node {\small{2}} (3,-0.4)
node {\small{3}}; \draw (4,-0.4) node {\small{4}} (5,-0.4) node
{\small{5}}; \draw[fill] (0,2) circle (2pt) (1,0) circle (2pt)
(5,0) circle (2pt); \draw (3,1) circle (2pt); \draw[thick] (0,2)
-- (1,0) -- (5,0); \draw (2,0) node {$*$} (4,0) node {$*$}; \draw
(0.2,2.2) node {\tiny{$w_0$}} (1.2,0.2) node {\tiny{$w_1$}}
(2.2,0.2) node {\tiny{$w_2$}}; \draw (3.2,1.2) node {\tiny{$w_3$}}
(4.2,0.2) node {\tiny{$w_4$}} (5.2,0.2) node {\tiny{$w_5$}};
\end{tikzpicture}
\begin{tikzpicture}[scale=0.6]
\draw (0,3) -- (0,0) -- (6,0); \draw (0,1) node {$-$} (2,0) node
{$|$} (4,0) node {$|$}; \draw (-0.2,1) node {\small{1}} (-0.2,2)
node {\small{2}}; \draw (1,-0.4) node {\small{1}} (2,-0.4) node
{\small{2}} (3,-0.4) node {\small{3}}; \draw (4,-0.4) node
{\small{4}} (5,-0.4) node {\small{5}}; \draw[fill] (0,2) circle
(2pt) (1,0) circle (2pt) (2,1) circle (2pt); \draw[fill] (4,1)
circle (2pt) (5,0) circle (2pt); \draw[thick] (0,2) -- (1,0) --
(5,0); \draw[dashed] (2,1) -- (4,1); \draw (3,0) node {$*  $};
\draw (0.2,2.2) node {\tiny{$w_0$}} (1.2,0.2) node {\tiny{$w_1$}}
(2.2,1.2) node {\tiny{$w_2$}}; \draw (3.2,0.2) node {\tiny{$w_3$}}
(4.2,1.2) node {\tiny{$w_4$}} (5.2,0.2) node {\tiny{$w_5$}};
\end{tikzpicture}
\end{tabular}
\caption{Left: Type I, Center: Type II, Right: Type III}\label{fig:3types}
\end{figure}

To explain Figure~\ref{fig:3types} featuring the three types of cones in $\mathcal{T}(Sev_n^2)$, we need to set some notations.
Any vector $w \in \R^{n+1}$ defines a marked subdivision
$\Pi_w$ of $\cA = \{ 0, \dots, n\}$ as follows.
Let $\Gamma_1,\ldots, \Gamma_s$ be
the lower faces of the convex hull of the set $\Gamma_w=\{(0,w_0),
\dots, (n,w_n)\}$, known as the Newton diagram of $w$.
Let $\sigma_1,\ldots,\sigma_s$ be the subsets
of $\cA$ obtained by projecting the points of $\Gamma_w$ in each
face $\Gamma_1,\ldots, \Gamma_s$ onto the first coordinate. These
subsets (or cells) determine a regular subdivision $\Pi_w=
\{\sigma_1, \ldots, \sigma_s\}$ of $\cA$. A point $j$ is a marked
point in the cell $\sigma_i$ if $(j, \val(w_j))$ is a point in the
relative interior of the face $\Gamma_i$. In that case, the cell
$\sigma_i$ is called a marked cell.
 Figure~\ref{fig:3types} depicts for $n=5$, the
Newton diagrams corresponding to weights $w \in \R^6$ in the
interior of one sample cone of each type. Note that for points $w$
in (the relative interior of) cones of type I, there are two
different marked segments (with one marking each) in $\Pi_w$. For
points $w$ in (the relative interior of) cones of type II, there
is a single segment in $\Pi_w$ with two markings (and the
exceptional small configurations
in~Definition~\ref{def:exceptional configurations} should be
excluded by Proposition~\ref{prop:exceptionalConfig}); these are
also cones that arise from the intersection of two maximal
dimensional cones of the tropical discriminant. Regular
subdivisions $\Pi_w$ associated with points $w$ in (the relative
interior of) cones of type III, have a single marked segment with
a single marking, but there is also a tie between values that do
not lie in the Newton diagram, that is, which is not visible from
the marked subdivision (thus called a ``hidden tie''). Moreover,
we will see in Definition~\ref{def:III} that there are arithmetic
ingredients which control these values. These cones arise from two
nodes with the same valuation and with further algebraic relations
on the initial coefficients. They don't have for the moment a
purely tropical explanation. We refer to
Definitions~\ref{def:I},~\ref{def:II} and~\ref{def:III} for the
precise conditions.

Our main result is the following:

\begin{theorem}\label{teo:main} Let $n \ge 4$. Then,
$\mathcal{T}(Sev_n^2)\subset \R^{n+1}$ equals as a set the
union of all cones of types $I$, $II$ and $III$.
\end{theorem}

The proof of Theorem~\ref{teo:main} follows from  a series of results. On one side,
Theorem~\ref{teo:firstInclusion} shows that $\mathcal{T}(Sev_n^2)$ is contained in
the union of the cones of types $I$, $II$ and $III$. On the other side,
Theorem~\ref{teo:Imain} proves that cones of type I lie in $\mathcal{T}(Sev_n^2)$,
Theorems~\ref{teo:IImaina} and~\ref{teo:IImainb} show that cones of type II lie in $\mathcal{T}(Sev_n^2)$,
and finally, Theorem~\ref{teo:SufficiencyIII} gives the last inclusion of cones of type III in
the tropicalization of the Severi variety.

The article is structured as follows, In Section~\ref{sec:homogeneous} we present some general
results concerning the tropicalization of homogeneous linear ideals over an infinite valuated ring
and we give a characterization using maximal minors of an associated matrix in Theorem~\ref{teo:matrix-by-minors}.
In Section~\ref{sec:cones}, some known results and notation are stated and a detailed description of cones of type $I$, $II$ and $III$ is given.
We prove in Section~\ref{sec:necessary} that $\mathcal{T}(Sev_n^2)$ is contained in the union of the cones of type $I$, $II$ and $III$,
while Section~\ref{sec:sufficient} deals with the converse inclusion. Finally, we include three separated appendices for the convenience to the reader.
The first one contains a technical result about linear spaces and circuits. The second one contains a set of results about minors of the matrices that
appear naturally when studying $Sev_n^2$. The last Appendix contains code that has been used to verify some claims along the text.

Note that our approach in Section~\ref{sec:homogeneous} is very general (independent of $n$ and $\delta$)
and could be used to describe the case of any number $\delta$ of nodes. However, we restricted our attention
to the particular case $\delta=2$, which highlights the new phenomena that
occur in type III cones and that already requires many technical results.
Our approach is used in Example~\ref{ex:Esterov} to clarify the two dimensional case worked out in~\cite{Esterov15}.

\section{Results on homogeneous linear ideals}\label{sec:homogeneous}

In this section we present some basic results about the tropicalization of linear homogeneous
ideals over a valuated field. We refer the reader to the book~\cite{MS14} for basic
results on valuations and tropicalizations.

Along the paper, $\K$ will denote an algebraically closed field
with a rank one non-archimedean valuation
\[\val :\K^*\rightarrow \R\]
such that both $\K$ and its residue field $K$ are of
characteristic zero. This is equivalent to the fact that integers have valuation $0$.
The value group $\Gamma$ is dense in $\R$ and possibly multiplying $\val$ by a constant,
we may assume that $\Gamma$ contains the rationals.
As $\K$ is algebraically closed, there exists
a split of the valuation. That is, a multiplicative subgroup
$\{t^a\ : \ a \in \val(\K^*)\}\subseteq \K^*$ isomorphic to $\Gamma$.

 If $b\in \K^*$ with $\val(b)=v$,
we will usually write $b=\beta t^v+*\in \K, \beta \in K^*$, to distinguish the
lowest term. The reader may assume for simplicity that $\K=\C\{\!\{t\}\!\}$,
the field of Puiseux series with its usual valuation, and $K=\C$.

Given a polynomial $g = \sum_{\alpha \in \N^{n+1}} g_\alpha x^\alpha \in \K[x_0, \dots, x_n]$ (or in the Laurent polynomial ring
$\K[x_0^{\pm 1}, \dots, x_n^{\pm 1}]$), the {\em tropicalization} $trop(g)$ of $g$ is the piecewise
linear function on $\R^{n+1}$ defined by:
\[trop(g)(w) \, = \, {\rm min} \{ \val(g_\alpha) + \langle w,\alpha \rangle \, : \, g_\alpha \neq 0\},\]
and its zero set $V(trop(g)) \subset \R^{n+1}$ is given by those $w$ for which the minimum in $trop(g)$ is attained
at least twice. The tropicalization $\T(I)$ of an ideal $I$ in the Laurent polynomial ring is defined as the intersection
of the zero sets $V(trop(g))$ for any nonzero $g \in I$.
The tropicalization $\T(X)$ of an algebraic subvariety $X$ of
the torus is the tropicalization of its ideal. Given an affine or projective irreducible variety,
its tropicalization is defined as the tropicalization of its intersection with the torus.
In fact, for any ideal $I$, its tropicalization $\T(I)$ can be given as a finite intersection, and any
finite subset of polynomials in $I$ which suffice to describe $\T(I)$ is called a tropical basis of $I$ (\cite{B+07}).


Lemma~\ref{lemma:minimal} below  tells us that any
linear homogeneous ideal has
a tropical basis formed by circuits, as in the well known case of
a trivial valuation.
For the convenience of the reader, we include
in the Appendix a Gr\"obner free proof, which is adapted from
Lemma~3.12 and Theorem~3.13 in~\cite{Tabera}.

We first recall the definition of circuits.

\begin{definition} Let $I$ be an homogeneous
linear ideal in $\K[x_0, \dots, x_n]$. A circuit in $I$ is a non-zero linear form $\ell \in
I$ with minimal support.
Let $M$ be a matrix in $\K^{d\times (n+1)}$. A circuit in $M$ is a
non-zero element $ r \in {\rm rowspan}(M)$ with minimal support.
\end{definition}
Note that there is a finite number of circuits in $I$ (up to a multiplicative
constant). Clearly, if we interpret the rows of
a matrix $M$ as coefficients of linear forms in variables $x_0,
\dots, x_n$, the circuits of $M$ coincide with the circuits of the
ideal $I(M)$ generated by these $d$ linear forms.

\begin{lemma}\label{lemma:minimal} Let $I \subset \K[x_0, \dots, x_n]$ be a homogeneous
linear ideal. Then
\[\mathcal{T}(I)=\bigcap_{\substack{\ell \in I \\ \ell \textrm{ circuit}}}V(trop(\ell)).\]
\end{lemma}

\begin{remark}\label{rmk:finitefield}
Lemma~\ref{lemma:minimal} holds under the weaker hypotheses that $\K$ is a valuated field
(not necessarily algebraically closed) with a rank one valuation
such that the residue field is infinite. In particular,
Theorem~\ref{teo:matrix-by-minors} below holds in the positive characteristic case and the $p$-adic case.

The hypothesis that the residue field is infinite cannot be
avoided. Let $\mathbb{F}_3$ be the field of three elements and
consider $\K=\mathbb{F}_3(t)$ with the usual valuation. Let $I$ be
the linear ideal generated by $x_1 - x_4 + x_6, x_2 - x_4 - x_6,
x_3 + x_4, x_5 + x_6$. Then $(0,0,0,0,0,0)\in trop(\ell)$, for
every linear form $\ell\in I$. However, there is no element in
$V(I)$ with coefficients in $\mathbb{F}_3(t)$ whose
tropicalization is $(0,0,0,0,0,0)$. On the other hand, in
$\overline{\mathbb{F}_3}(t)$, we have that the point
$(2+\alpha,1+\alpha,2\alpha,\alpha,2,1)\in V(I)$, with
$\alpha^2+1=0$. Its tropicalization is the desired point
$(0,0,0,0,0,0)$.
\end{remark}

Now we want to characterize the circuits of a matrix $M\in
\K^{d\times(n+1)}$ in terms of its minors.

\begin{notation}\label{not:M_J} Let $M$ be a matrix
in $\K^{d\times(n+1)} $. Let $J= \{i_1,
\dots, i_{s}\}\subseteq \{0, \dots, n\}$ with $ i_1<\dots <i_{s}$.
If we denote $C_0, \dots, C_n$ the ordered columns of the matrix
$M$, then $M_J$ is defined as the matrix in $\K^{d\times s}$ with
columns $C_{i_1}, \dots, C_{i_{s}}$. If $s=d-1$, then $r_J\in {\rm
rowspan}(M)$ is the vector in $\K^{(n+1)}$ defined by
$r_{J,k}= (-1)^{\mu(k,J)} \det(M_{J \cup \{k\}})$, where
$\mu(k,J)$ is the sign of the permutation of $J \cup \{k\}$ which takes $k$ followed by
the ordered elements of $J$ to the ordered elements of $J \cup \{k\}$,  for all $k \in \{0,
\dots, n\}$.
\end{notation}

Without loss of generality, $M$ can be assumed to be of maximal
rank. The following lemma is straightforward

\begin{lemma}\label{lemma:circuits}
Let $M\in \K^{d\times (n+1)}$ be a matrix of rank $d$ and
$J\subset \{0, \dots, n\}$ such that $\#J=d-1$ and
rank$(M_J)=d-1$. Then $r_J\in\K^{(n+1)} $ is a circuit of $M$.
Moreover, up to multiplicative constant, these are all the
circuits of $M$ (possibly repeated).
\end{lemma}

We can now describe precisely the tropicalization
of the kernel of $M$.

\begin{theorem}\label{teo:matrix-by-minors}
Let $M\in \K^{d\times (n+1)}$ be a matrix of rank $d$
and  $w\in
\mathbb{R}^{n+1}$. Then $w\in
\mathcal{T}(ker(M))$ if and only if, for
every $J\subseteq \{0,\ldots, n\}$ with $d-1$ elements such that
$rank(M_J)=d-1$, the minimum of
\begin{equation} \label{eq:minors}
\left\{\val(\det(M_{J \cup \{k\}}))+w_k\quad |\quad  k \in \{0, \dots,
n\}\right\}
\end{equation}
is attained at least twice.
\end{theorem}

\begin{proof}
Given $M$, let $I(M)$ be the homogeneous linear ideal generated by
the linear forms associated with the elements in ${\rm
rowspan}(M)$. For every linear form $\ell(x) = \sum_{i=0}^n \ell_i
x_i$ in $I(M)$, $w \in V(trop(\ell))$ if and only if the initial
form
\[\sum_{\substack{\val(\ell_i)+w_i
\\ \rm{minimal}}} \ell_ix_i \] is not a monomial, which is equivalent
to $\min\{\val(\ell_0)+w_0, \dots, \val(\ell_n)+w_n\}$ being
attained at least twice.

Because of the Lemma~\ref{lemma:minimal}, $ \mathcal{T}({\rm
ker}(M))=\mathcal{T}(I(M))=\bigcap V(trop(\ell))$ where the
intersection is over all circuits in $I(M)$. Then, $w \in
\mathcal{T}({\rm ker}(M))$ if and only if for every circuit
$r=(r_0, \dots, r_n) \in {\rm rowspan}(M)$,  there are at least
two elements in the set $\{i \in\{0, \dots, n\} \ : \ \val
(r_i)+w_i=\min \{\val (r_0)+w_0, \dots, \val(r_n)+w_n\}\}.$ The
result follows from Lemma~\ref{lemma:circuits} and the fact that
$\val(-r_i)=\val(r_i)$.
\end{proof}

As an example, we use Theorem~\ref{teo:matrix-by-minors} to
present a new explanation of the cones in the tropicalization of
the set of bivariate polynomials with two nodes given in Example
3.9 of \cite{Esterov15}.

\begin{example} \label{ex:Esterov} Consider the set
$\mathrm{A}=\{(0,0),(1,0),(1,1),(0,1),(-1,0),(0,-1)\}$ and the
variety $Sev^2_{\mathrm{A}}$, given as the (Zariski) closure of
the set
$$\{ c \in (\K^*)^6 \, : \,
f(x,y)= c_1+c_2x+c_3xy+c_4y+c_5\frac{1}{x}+c_6\frac{1}{y} \mbox{
has two nodes in } (\K^*)^2 \}.$$

Consider the subvariety $S \subset Sev^2_{\mathrm{A}}$ given by the closure of those
  $c=(c_1, \dots, c_6)$  with a node at the fixed point $(1,1)$ plus
  another node that we call
$(b_1,b_2) \in (\K^*)^2 \setminus \{(1,1)\}$.
Then, such a  $c$ satisfies $c \in \mbox{
ker}(M'), $ where
$${\small M' = {\left(\begin{array}{ccccrr}
   1   &     1    &     1      &    1     &   1 &  1 \\
   0   &     1    &     1      &    0     & -1 &  0 \\
   0   &     0    &     1      &    1     &  {}0 & -1 \\
   b_1b_2 & b_1^2b_2 & b_1^2b_2^2 & b_1b_2^2 & b_2 & b_1 \\
   0   &   b_1^2  &  b_1^2b_2  &    0     & -1 &  0 \\
   0   &     0    &  b_1b_2^2  &   b_2^2  &  {}0 & -1 \end{array}\right)}}.$$
As $\det(M')=b_1b_2(b_1-1)^2(b_2-1)^2(b_1b_2-1)$,
$M'$ has a non-trivial kernel if and only if $b_1=1$ or
$b_2=1$ or $b_2=\frac{1}{b_1}$. In these cases, it can be checked that rank$(M')=5$.

We will call $M \in \K^{5\times 6}$ the submatrix of $M'$ of maximal rank obtained
by removing the 5-th row if $(b_1,b_2) \in \{(-1,1),(-1,-1)\}$, the 6th row if
$(b_1,b_2)=(1,-1)$, and the 4-th row for very other $(b_1,b_2)$
where the determinant vanishes.
Applying Theorem~\ref{teo:matrix-by-minors} we have that $w\in
\mathcal{T}(ker(M))$ if and only if, for every $J\subseteq
\{1,\ldots, 6\}$ with $4$ elements and $rank(M_J)=4$, the minimum
of
\begin{equation}\label{eq:esterov}
\{\val(\det(M_{J \cup \{k\}}))+w_k\quad |\quad  k \in
\{1,\dots,6\}\setminus J\}
\end{equation} is attained at least twice. Because of the size of this example,
this can be translated into de fact that for both
$k_1,k_2 \in \{1, \ldots, 6\}\setminus J$, $\val(\det(M_{J \cup \{k_1\}}))+w_{k_1}=\val(\det(M_{J \cup \{k_2\}}))+w_{k_2}$.
Based on this theorem, we can easily discard $b_1=1$ or $b_2=1$.
For instance, if $b_2=1$, taking $J=\{1,2,3,4\}$ we obtain that the minimum in
Equation~(\ref{eq:esterov}) is only attained once at $k=6$.

We now describe all the cones in $\mathcal{T}(Sev^2_{\mathrm{A}})$
when $b_2=\frac{1}{b_1}$. As before, the case $(b_1,b_2)=(-1,-1)$
can be discarded by taking the set $J=\{2,3,4,5\}$ for which we obtain that the minimum in
Equation~(\ref{eq:esterov}) is only attained once at $k=1$.

For every $(b_1,b_2)\in(\K^*)^2-\{
(-1,-1),(1,1)\}$, let $M_i$ be the $5\times 5$ matrix obtained by
removing the $i$-th column from $M$. Then, we can compute
$$\det(M_1)= -\det(M_3)=\frac{(b_1-1)^2(b_1+1)^2}{b_1^2},\
\det(M_4)=\det(M_5)= - \frac{(b_1-1)^2(b_1+1)}{b_1} $$ $$\mbox{
and } \det(M_2)=-\det(M_6)= -\frac{(b_1-1)^2(b_1+1)}{b_1^2}.$$
Since none of them vanishes at $(b_1,b_2)$, every $J \subset \{1,
2,3,4,5,6\}$ with 4 elements satisfies that rank$(M_J)=4$.

When $\val(b_1)\ne 0$, using Theorem~\ref{teo:matrix-by-minors}
and considering the valuations of the determinants of the matrices $M_i$,
we obtain that $w \in \mathcal{T}(ker(M))$ if and only
if one of the following sets of equations is verified
$$ w_1=w_2=w_3=w_6<w_4=w_5 \mbox{ if } \val(b_1)>0$$
$$ w_1=w_3=w_4=w_5<w_2=w_6 \mbox{ if } \val(b_1)<0.$$
These points induce the subdivision of
the convex hull of $\mathrm{A}$ in Figure~\ref{fig:one}.
\vspace{-0.25cm}{\begin{figure}[hbtp]
\begin{tikzpicture} \draw[thin] (-1,0) -- (-0.5,0);
\draw[thin] (0.5,0) -- (1,0); \draw[thin] (0,-1) -- (0,-0.5);
\draw[thin] (0,0.5) -- (0,1); \draw[fill] (0,0) circle (2pt);
\draw[fill] (0.5,0) circle (2pt); \draw[fill] (0,0.5) circle
(2pt); \draw[fill] (-0.5,0) circle (2pt); \draw[fill] (0,-0.5)
circle (2pt); \draw[fill] (0.5,0.5) circle (2pt); \draw[thick]
(-0.5,0) -- (0,-0.5) -- (0.5,0) -- (0.5,0.5) -- (0,0.5) --
(-0.5,0); \draw[thick] (0.5,0.5) -- (0,0) -- (-0.5,0);
\draw[thick] (0,0) -- (0,-0.5);
\end{tikzpicture}
\caption{$\val(b_1)\ne 0$}\label{fig:one}
\end{figure}
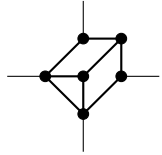}

 When $\val(b_1)= 0$ and $b_1=\beta+ *$ where $\beta\ne -1$, we
obtain points in $\langle(1,1,1,1,1,1)\rangle$.

When $\val(b_1)= 0$ and $b_1=-1+ *$ ($b_1\ne -1)$, we obtain
points $w \in (\K^*)^6$ that induce the regular subdivision in
Figure~\ref{fig:no_hidden_curve}. \vspace{-0.25cm}{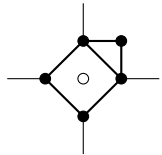
\begin{figure}[hbtp]
\begin{tikzpicture} \draw[thin] (-1,0) -- (-0.5,0);
\draw[thin] (0.5,0) -- (1,0); \draw[thin] (0,-1) -- (0,-0.5);
\draw[thin] (0,0.5) -- (0,1); \draw (0,0) circle (2pt);
\draw[fill] (0.5,0) circle (2pt); \draw[fill] (0,0.5) circle
(2pt); \draw[fill] (-0.5,0) circle (2pt); \draw[fill] (0,-0.5)
circle (2pt); \draw[fill] (0.5,0.5) circle (2pt); \draw[thick]
(-0.5,0) -- (0,-0.5) -- (0.5,0) -- (0.5,0.5) -- (0,0.5) --
(-0.5,0); \draw[thick] (0.5,0) -- (0,0.5);
\end{tikzpicture}
\caption{$\val(b_1)=0, \beta= -1$}\label{fig:no_hidden_curve}
\end{figure}}

This subdivision seems to be associated with the existence
of one node, but the existence of a second node imposes a ``hidden
tie'' from the relations $$ w_1=w_3>w_2=w_4=w_5=w_6.$$
Thus, we
 get the subdivision that Esterov
represents with the dotted segment representing the hidden tie depicted in Figure~\ref{fig:hidden_curve}.
\begin{figure}[hbtp] \label{fig:3}
\begin{tikzpicture} \draw[thin] (-1,0) -- (-0.5,0);
\draw[thin] (0.5,0) -- (1,0); \draw[thin] (0,-1) -- (0,-0.5);
\draw[thin] (0,0.5) -- (0,1); \draw[fill] (0.5,0) circle (2pt);
\draw[fill] (0,0.5) circle (2pt); \draw[fill] (-0.5,0) circle
(2pt); \draw[fill] (0,-0.5) circle (2pt); \draw[fill] (0.5,0.5)
circle (2pt); \draw[thick] (-0.5,0) -- (0,-0.5) -- (0.5,0) --
(0.5,0.5) -- (0,0.5) -- (-0.5,0); \draw[thick] (0.5,0) -- (0,0.5);
\draw[dashed] (0,0) -- (0.5,0.5);
\end{tikzpicture}
\caption{Hidden tie}\label{fig:hidden_curve}
\end{figure}
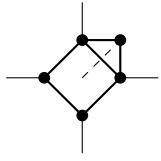

Indeed, ${\mathcal T}(S)$ is a tropical linear space of dimension
two. Now, we add the action of the linear space ${\rm
rowspan}(A)$, where $A$ is the $2\times 6$ matrix of exponents of
our configuration read in the second and third rows of $M'$ to get
$\mathcal{T}(Sev^2_{\mathrm{A}})$ as the union of the two cones
giving the subdivisions and hidden tie presented in Example 3.9 of
\cite{Esterov15}. These cones are explicitly $\mathcal
C_\diamond$, with $\diamond$ equal to $\le$ and to $\ge$:
\[\mathcal C_\diamond = \{ w \in \R^6 \, : \, w_1-w_3=w_5-w_4=w_6-w_2, 4w_1 \, \diamond \, w_2+w_4+w_5+w_6 \}.\]
In fact, $Sev^2_{\mathrm{A}}$ is in this case a toric variety, defined by the equations
$$c_2 c_5 - c_4 c_6=0, c_1 c_4 - c_3
c_5=0, c_1 c_2 - c_3 c_6=0.$$
\end{example}

\section{Defining the cones}\label{sec:cones}

In this section, we describe the cones  in $\mathcal{T}(Sev_n^2)$.

\subsection{Notation and basic facts}\label{sec:basic-facts}

Let $K$ be an algebraically closed field of characteristic zero, $n \in \N_{\ge 4}$, and
let $\mathcal U \subset  K^{n-1}$ be the open set defined by
\[\mathcal U \, = \, \{(a, b, a_1,\dots, a_{n-3})\in (K^*)^{n-1} \, : \, a,b, a_1, \dots, a_{n-4} \text{ are all different}
 \}.
\]
The Severi variety $Sev_n^2(K)$   is
the Zariski closure of the image of the map $\varphi_K$ defined by:
\[\begin{array}{rcl}
\mathcal U & \longrightarrow & K^{n+1} \\
 (a_1, a_2, a_3,\dots, a_{n-1}) & \mapsto & (c_0, \dots, c_n),
\end{array}\]
where $c= (c_0, \dots, c_n)$ denotes the vector of coefficients of the polynomial
\begin{equation}\label{eq:fc}
 f \, = \, \sum_{i=0}^n c_i \, x^i \, = \, a_{n-1} (x-a_1)^2 (x-a_2)^2 \prod_{i=3}^{n-2} (x-a_i).
\end{equation}
Thus, $I(Sev_n^2(K))$ is defined over $\Q$.

Let $\K$ be an algebraically closed valuated extension of $K$, with residue field $K$.
For any $c$ in the torus $(\K^*)^{n+1}$, we denote by $\val(c) = (\val(c_0),\dots, \val(c_n))$ and
we will identify $c$ with the univariate polynomial $f$ with coefficients $c$.
The image $\val((\K^*)^{n+1}) \subset \R^{n+1}$ will be denoted by ${\rm {Im}}_{\val}$.
Also, for any $w \in  {\rm {Im}}_{\val}$ we denote by $\K[x]_w$ the set of polynomials $f$ with $w=\val(c)$.

By Theorems~3.2.2 and ~3.2.4 in \cite{MS14},
$$\mathcal{T}(Sev_n^2(K)) \, = \, \mathcal{T}(Sev_n^2(\K)) \, = \, {\rm closure}
\{ \val(c), c \in {\rm Im}(\varphi_{\K}) \cap (\K^*)^{n+1} \}.$$
As the field is not essential,  we will denote in what follows the
tropicalization of our Severi varieties by $\mathcal{T}(Sev_n^2)$.

\begin{definition} \label{def:matrixM(x)}
Let $\mathbf{M}\in (\Z[x])^{4\times (n+1)}$ be the matrix
$$\mathbf{M}=\begin{pmatrix} 1 & 1 & 1 & \dots & 1 \\ 0
& 1 & 2 & \dots & n \\ 1 & x & x^2 & \dots & x^n \\ 0 & x & 2x^2 &
\dots & nx^n
\end{pmatrix}.$$
\end{definition}

Let $J=\{i_1, i_2, i_3, i_4\}$ be a subset of $ \cA$ with
$i_1<i_2<i_3<i_4$. Let $\mathbf{M}_J$ be the submatrix of $\M$ as in
Notation~\ref{not:M_J}, and denote  $D_J = \det(\M_J)\in
\Z[x]$.

If $f\in \K[x]$ has double roots $a_1, a_2 \in \K^*$ with
$\val(a_1)\leq \val(a_2)$, then $f(a_1x)$ has double  roots $1$,
$b= a_2/a_1$ and $\val(b)\geq 0$. Hence, we can assume that $f$
has nodes at $1$ and $b$, where $b \not=0,1$ and $\val(b)\ge 0$.
Clearly, $f$ has nodes at $1$ and $b$ if and only if its vector of
coefficients $c$ lies in the kernel of the matrix $\mathbf M(b)$.

Expression~(\ref{eq:minors}) in Theorem~\ref{teo:matrix-by-minors} explains why
it is not enough to consider the combinatorial information given
by $w \in \mathcal{T}(ker(\M(b))$, as the picture is completed by
the valuations of the maximal minors of $\M(b)$, which depend on
arithmetic conditions on the chosen columns $J$ and algebraic
conditions on $b$.  Many technical properties of these matrices
and, particularly, of its $4\times 4$ minors are studied in
Appendix~\ref{app:PropertiesM}.

\medskip

\subsection{Cones of types $I$, $II$ and $III$}

We now give the precise definitions of cones of type I, II and III.
The tropical discriminant $\mathcal{T}(\Delta_n)$, that is, the
tropicalization of the variety of polynomials $f\in \K[x]$ of
degree $n$ with a double
root is known to consist of maximal dimensional cones of the secondary fan of the configuration $\mathcal{A}$.
Every cone of $\mathcal{T}(\Delta_n)$ corresponds naturally with a
marked subdivision of  $\mathcal{A}$. The maximal cones of
$\mathcal{T}(\Delta_n)$ (which have dimension $n$) correspond to the subdivisions of
$\mathcal{A}$ that have exactly one marked cell with exactly one
marked point. We refer to \cite{DFS07}, \cite{GKZ}, \cite{DT12}
for basic results on tropical discriminants.
On the other hand, we will see that not every $n-1$ dimensional cone given by
the intersection of two maximal cones in
$\mathcal{T}(\Delta_n)$ belongs to $\mathcal{T}(Sev_n^2)$, as there are exceptional configurations
(see Proposition~\ref{prop:exceptionalConfig}).

\begin{remark} \label{rem:1or2markedPoints}
Let $w\in\mathcal{T}(Sev_n^2)$. As $w$ is in the discriminant
variety, the subdivision induced by $w$ has at least one marked
cell. Moreover, if $w$ is a point in the interior of a maximal
cone of $\mathcal{T}(Sev_n^2)$, the Newton diagram has at most two
marked points because the maximal cones have codimension 2 and
every marked point in the Newton diagram imposes a linear
restriction that lowers the dimension.
\end{remark}

Given a configuration $J=\{i_1,\dots,i_k\}$ and $s \in \Z$, we
will use the notation $J(+s)$ for the set $\{i_1+s,\dots,i_k+s\}$
(called a {\em translation}) and $J(\times s)$ (when $s \in \N$)
for the set $\{s \, i_1,\dots,s \, i_k\}$.

\begin{definition}\label{def:I}
Consider a regular subdivision $\Pi$ with exactly two different
marked cells (each with only one marked point). Let $C_{\Pi}$ be
the closure of the set of $w\in \R^{n+1}$ with $\Pi_w=\Pi$. Then,
$C_{\Pi}$ is a cone and we say that it is a cone of \emph{type} $I$.
\end{definition}

\begin{example}\label{ex:typeI}
Let $n=5$, $\K=\C\{\!\{t\}\!\}$ with the usual valuation and
$$ f = t(x-1)^2(x-t)^2(x-t^{-1}).$$
Them $f \in \K[x]_w$, where $w=(2,1,0,0,0,1)$.
The Newton diagram of $w$ is displayed on the Left in Figure~\ref{fig:3types}.
Thus, $w$ lies in the
interior of the cone of type $I$ induced by the marked subdivision
$\{\{0,1,2\},\{2,3,4\},\{4,5\}\}$.
\end{example}

To introduce cones of type $II$, we need to specify what we understand by an
exceptional configuration.

\begin{definition}\label{def:exceptional configurations}
A marked cell  $\{i_1,i_2,i_3,i_4\}$ with two marked points  is
called an exceptional configuration if it is one of the following:
$\{0,1,2,3\}$, $\{0,1,2,4\}$, $\{0,2,3,4\}$, $\{0,3,4,6\}$ and
$\{0,2,3,6\}$.
\end{definition}

We now define type $II$ cones:

\begin{definition}\label{def:II}
Consider a regular subdivision $\Pi$ with exactly one marked cell $\{i_1,i_2,i_3,i_4\}$
with two marked points, which
 is not the image under a translation of an
exceptional configuration. Let $C_{\Pi}$ be the closure of the set
of $w\in \R^{n+1}$  with $\Pi_w=\Pi$. Then, $C_{\Pi}$ is a cone and we
say that it is a cone of \emph{type} $II$.
\end{definition}

\begin{example}\label{ex:typeII}Let {$n=5$,} $\K=\C\{\!\{t\}\!\}$ with the usual valuation
and
$$ f= (x-1)^2(x+2+\sqrt3+t)^2(x-t^2).$$
Then $f \in \K[x]_w$, with $w=(2,0,0,1,0,0)$. The Newton Diagram
of $w$ is displayed in the Center in Figure~\ref{fig:3types}.
Hence, $w$ lies in the interior of the cone of type $II$ induced
by the marked subdivision $\{\{0,1\}, \{1,2,4,5\}\}$.

Note that if the double root different from 1 has valuation 0 and
independent coefficient other than $-2\pm \sqrt{3}$ and -1, the
vector of valuations of the coefficients of the polynomial is in a
cone of type $II$, but not in its interior. See
Remark~\ref{rem:exTypeIIminors} for an expanded explanation of the
choice of $-2\pm \sqrt{3}$.
\end{example}

Type $III$ cones are \emph{not} cones of the secondary fan of $\cA$ as they also show additional arithmetic constraints:

\begin{definition}\label{def:III}
Let $\sigma = \{i_1, i_2, i_3\}$ and $\tau=\{j_1, j_2\}$ with
$i_1,i_2,i_3,j_1,j_2$ different points of $\cA$, which satisfy the
following restrictions: $g=gcd(i_3-i_1,i_2-i_1) > 1$, and there
exists a divisor $ d> 1$ of $g$ which does \emph{not} divide
$j_1-i_1$ nor $j_2-i_1$. Moreover, let $\Pi$ be a regular
subdivision of $\cA$ with $\sigma$ as the unique marked cell. Let
$C= C(\Pi, \tau)$ be the closure of the set of $w\in \R^{n+1}$
which verify:
\begin{itemize}
 \item[i)] $\Pi_w=\Pi$.
\item[ii)]  Let $\eta$ be the interior normal of the
lower facet of $\Gamma_w$ corresponding to the marked cell. The minimum value of the scalar products $\langle \eta, (j,w_j)
\rangle$ over all $j$ with $j-i_1 \not\equiv 0 \mod d$,
is attained at $j_1, j_2$.
\end{itemize}
Then, $C$ is a cone and we say it is a cone of \emph{type}
$III$ with a \emph{hidden tie} in $\tau$.
\end{definition}

\begin{remark}\label{rem:notprime}
 Note that it is not enough to ask $d$ to be a prime number in Definition~\ref{def:III}. Consider for instance
 the case $\sigma = \{0, 4,8\}$, $\tau=\{2,5\}$, $g=4$.  As we require that $d$ does not divide $2-0$ and $5-0$,
 $d=4$ is the only possible value of a divisor $d>1$ of $g$ to get a cone of type III.
\end{remark}

\begin{example}\label{ex:typeIII}
Let $n=5$, $\K=\C\{\!\{t\}\!\}$ with the usual valuation and
$$ f = (x-1)^2(x+1+t)^2(x-t^2).$$ Then,
$f \in \K[x]_w$, where $w=(2,0,1,0,1,0)$. We have displayed the
Newton Diagram of $w$ and the \emph{hidden tie} on the Right in
Figure~\ref{fig:3types}. Thus, $w$ lies in the cone of type $III$
induced by the marked subdivision $\Pi=\{\{0,1\},\{1,3,5\}\}$ and the
hidden tie in $\tau=\{2,4\}$.
\end{example}

\section{Necessary conditions}\label{sec:necessary}

In this section, we will show that $\mathcal{T}(Sev_n^2)\subset \R^{n+1}$ is contained in the
union of all cones of types $I$, $II$ and $III$, which proves one of the inclusions in the statement
of Theorem~\ref{teo:main}:

Given $w \in {\rm Im}_{\val}$, we first state some results that
relate the roots of a polynomial $g\in \K[x]_w$, and roots of its
associated residual polynomials in $K[x]$. For an element $a$, we
denote by mult$(a,g)$ the multiplicity of $a$ as a root of $g$.

\begin{definition}\label{def:initial}
Given $w \in\R^{n+1}$, let $\sigma$ be a cell in $\Pi_w$. For $g =
\sum_{i=0}^n d_ix^i \in \K[x]_w$, write $d_i=\delta_it^{w_i}+*$
(so all $\delta_i\neq 0$). We say that $g_{\sigma}=\sum_{i\in
\sigma}\delta_ix^i\in K[x]$ is the residual polynomial of $g$ with
respect to $\sigma$.
\end{definition}

The following proposition is a
generalization of the classical Newton-Puiseux theorem on the valuation of the
roots of a polynomial when the polynomial has multiple roots. It is a direct consequence
of results in \cite{TesisTabera} and \cite{DDR13}.

\begin{proposition} Let $g=\sum_{i=0}^n d_ix^i$ be a
polynomial in $\K[x]$ with $d_0\neq 0 \neq d_n$ and let $a_1, \dots,
a_n$ be the roots of $g$ (repeated according to their
multiplicities). Let $w\in \R^{n+1}$ be the vector $w=\val (d)$.
Using notation from Definition~\ref{def:initial},
\begin{itemize} \label{prop:ResultadosAliciaLuis}
\item [i)] There exists $\sigma$ a cell of $\Pi_w$ such that
the lower facet in the Newton Diagram of $w$ that induces $\sigma$
has lattice length $\ell$ and slope $-v$, if and only if the set
$J_{\sigma}= \{j \in \{1, \dots, n\} \ : \ \val (a_j)=v\}$ has
cardinal $\ell>0$.
\item [ii)] Denote $a_j=\alpha_jt^v+*$ for all $j \in J_{\sigma}$.
Then $\{\alpha_j \ : \ j \in J_{\sigma}\}$ is the set of nonzero
roots of $g_{\sigma}$ and, for all $ j \in J_{\sigma}$, ${\rm
mult}(\alpha_j,g_{\sigma})=\sum{\rm mult}(a_{k},g)$ where the sum
is over all $k\in J_{\sigma}$ such that $\alpha_{k}=\alpha_j$.
\item [iii)] For all $j \in J_{\sigma}$,
both the multiplicity of $a_j$ as a root of $g$ and the
multiplicity of $\alpha_j$ as a root of $g_{\sigma}$ are strictly
smaller than the cardinal of the cell $\sigma$.
\end{itemize}
\end{proposition}
\begin{proof}
Items i) and ii) were proved in Proposition~1.8 and
Proposition~1.9 on \cite{TesisTabera} respectively.

To see item iii), Theorem~5.6 in \cite{DDR13} states in the case
of univariate polynomials that, if we fix $ \kappa \in \N$, $g$
has a root of multiplicity at least $\kappa+1$ if and only if
there exists $w \in \R^{n+1}$ with the following property: for any
$J \subset \cA$ with $\# J\le \kappa$, the minimum of the scalar
products $\langle w, \cdot \rangle$ is attained at least twice on
$\cA \setminus J$. Using this result and item ii), if there is a
root $a_j$ of $g$ of multiplicity $\tau$, $\alpha_j$ is a root of
$g_\sigma$ of multiplicity at least $\tau$, and therefore removing
any subset $J\subset \sigma$ of $\tau$ elements, $\sum_{i \in
\sigma-J} \delta_ix^i$ cannot be a monomial. It follows that the
the cardinal of $\sigma$ has to be at least $\tau+1$.
\end{proof}

As a straightforward consequence of
Proposition~\ref{prop:ResultadosAliciaLuis}, we see that:

\begin{corollary}
Let $w\in \R^{n+1}$ be a point such that the subdivision $\Pi_w$
has only one marked cell of lattice length 3. Then, $w$ does not
lie in $\mathcal{T}(Sev_n^2)$.
\end{corollary}

Another consequence of Proposition~\ref{prop:ResultadosAliciaLuis}
is the following:

\begin{corollary} \label{corol:differentResidue} Let $f
\in \K[x]_w$, with $w$ an element in the interior of a cone of
type $I$, $II$ or $III$. Let $a= \alpha t^{\val (a)}+*,b=\beta
t^{\val (b)}+*$ be two different multiple roots of $f$. If
$\val(a)=\val(b)$, then $\alpha \neq \beta $.
\end{corollary}
\begin{proof}
Since $\val(a)=\val(b)$, by item $ii)$ of
Proposition~\ref{prop:ResultadosAliciaLuis}, both $\alpha $ and
$\beta$ are multiple roots of the residual polynomial associated
to the marked cell of slope $-\val(a)$ in the Newton diagram of
$w$. Moreover, if $\alpha =\beta$, then by the same item it is
root of multiplicity at least four. By item $iii)$, this means
that the corresponding marked cell needs to have at least five
points (or equivalently, three marked points). But then, $w$
cannot be in the interior of a cone of type $I$, $II$, $III$,
which is a contradiction. Then, $\alpha \ne \beta$.
\end{proof}

\begin{remark} Given $f\in \K[x]$, we have already remarked that
we assume that one of its roots is equal to $1$. Moreover, by item
$i)$ in Proposition~\ref{prop:ResultadosAliciaLuis}, this means we
can assume a desired lower facet from the Newton Diagram to have
slope $0$. Multiplying $f$ by an appropriate constant in $\K^*$,
we can moreover assume without loss of generality that
$\val(c_i)\ge 0$ for all coefficients of $f$ and $\val(c_j)=0$ for
every $j\in \cA$ that lie in the cell associated to the lower
facet of slope $0$.
\end{remark}

In the definition of type $II$ cones, we explicitly excluded
subdivisions whose marked cell is a translation of an exceptional
configuration. We now show that these subdivisions with
exceptional configurations are not in $\mathcal{T}(Sev_n^2)$.

\begin{proposition}\label{prop:exceptionalConfig}
Let $\sigma=\{i_1,i_2,i_3,i_4\}$ be the translation of an
exceptional configuration. Let $w\in {\rm Im}_{\val}$ be a point
such that $\sigma$ is the only marked cell from the subdivision of
$\cA$ induced by $w$. Then $w$ is not in $\mathcal{T}(Sev_n^2)$.
\end{proposition}
\begin{proof}
As the lattice length of the marked cell has to be at least 4, any
$w$ such that the induced subdivision has as the only marked cell
a translation of $\{0,1,2,3\}$ is not in $\mathcal{T}(Sev_n^2)$.

As we will consider the nonzero roots of residual polynomials with
support in $\sigma$, without loss of generality we can assume
$\sigma$ an exceptional configuration. Also, we can assume that
the lower facet associated to $\sigma$ has slope 0. Suppose there
is a polynomial $f \in \K[x]_w\cap Sev_n^2$ and $1,b$ are its
double roots. As $\sigma$ is the only marked cell, $\val(b)=0$ and
$b=\beta+*$ where, as in the proof of
Corollary~\ref{corol:differentResidue}, $\beta\neq 1$.

If $\sigma = \{0,1,2,4\}$ and
$f_{\sigma}(x)=\gamma_4(x-1)^2(x-\beta)^2= \sum\limits_{i \in
\sigma}\gamma_ix^i$, then $\gamma_3=\frac{\gamma_1}{\beta}$ and
moreover, $\gamma_3=0$ and $\gamma_1\ne 0$, which is a
contradiction.

If $\sigma = \{0,2,3,6\}$ and
$f_{\sigma}(x)=\gamma_6(x-1)^2(x-\beta)^2(x^2+\delta
x+\epsilon)=\sum\limits_{i \in \sigma}\gamma_ix^i$, as
$\beta\gamma_2\gamma_3$ is a linear combination with coefficients
in $\Q[\beta,\delta,\epsilon]$ of the coefficients of the
monomials $x,x^4,x^5$ (see
Computation~\ref{comp:exceptionalConfig}),
$\beta\gamma_2\gamma_3=0$ which is also a contradiction.

Having ruled out these two configurations and considering the
polynomial $g\in \K[x]\cap Sev_n^2$ define as
$g(x)=x^{n}f(x^{-1})$, we discard the remaining exceptional
configurations because $\{0,2,3,4\}= \{4-i \ : \ i \in
\{0,1,2,4\}\}$ and $\{0,3,4,6\}=\{6-i\ : \ i \in \{0,2,3,6\}\}$.
\end{proof}

We now provide another technical result related to cones of type
$III$.

\begin{proposition} \label{prop:unityRoot} Let $n\ge 4$ and
$w \in \mathcal{T}(Sev_n^2)$ such that the subdivision induced by
$w$ has only one marked cell $\sigma=\{i_1,i_2,i_3\}$. Let $f\in
Sev_n^2\cap \K[x]_w$ and $a= \alpha t^v +{\rm *}, b=\beta t^v+ *
\in \K^*$ be multiple roots of $f$. Then $\frac{\beta}{\alpha}\ne
1$ is a root of unity of order that divides
$\gcd(i_2-i_1,i_3-i_1)$.
\end{proposition}
\begin{proof} We can assume as before that $a=1$, and
$i_1=0$. Let
$f_{\sigma}=\gamma_0+\gamma_{i_2}x^{i_2}+\gamma_{i_3}x^{i_3}$ be
the residual polynomial of $f$ associated to the marked cell
$\sigma$. Then by Proposition~\ref{prop:ResultadosAliciaLuis},
$1,\beta$ are double roots of $f_{\sigma}$ and $\beta \ne 1$. By
evaluating we obtain the following equalities:
$$\begin{matrix} i) & \gamma_0+\gamma_{i_2}+\gamma_{i_3}=0. & iii)
 &  \gamma_0+\gamma_{i_2}\beta^{i_2}+
\gamma_{i_3}\beta^{i_3}=0.\\
ii)  & i_2\gamma_{i_2}+i_3\gamma_{i_3}=0. & iv)
 &  i_2\gamma_{i_2}\beta^{i_2}+
i_3\gamma_{i_3}\beta^{i_3}=0.
\end{matrix}$$
By equalities $i)$ and $ii)$ we have
$\gamma_0=(\frac{i_3}{i_2}-1)\gamma_{i_3}$ and
$\gamma_{i_2}=-\frac{i_3}{i_2}\gamma_{i_3}.$ Replacing in equality
$iv)$ and simplifying we obtain $\beta^{i_2}=\beta^{i_3}$.
Finally, using equality $iii)$ we obtain $\beta^{i_3}=1$. It
follows that $\beta$ is an $d$-th primitive root of unity where
$d\ne 1$ and $d$ divides $\gcd(i_2,i_3)$.
\end{proof}

Finally we can prove that, for $w \in \R^{n+1}$ to be a point in
$\mathcal{T}(Sev_n^2)$, is a necessary condition that $w$ be a
point in the union of all cones of type $I$, $II$ or $III$.
We start by defining an open dense subset $\mathcal V$ of $Sev_n^2(\K)$.
Consider the open dense subset ${\mathcal U}'$ of $\mathcal U$ given by those $(a,b,a_1,\dots, a_{n-3})$
such that $D_J(a/b) \neq 0$, for any $J \subset \cA$ of cardinal 4. Note that $D_J(a/b) \neq 0 $ if and only if
$D_J(b/a) \neq 0$ by Lemma~\ref{lemma:i=0}. Moreover, ${\mathcal U}'
\neq \emptyset$
by Lemma~\ref{lemma:B3}. We call $\mathcal V = \varphi_\K({\mathcal U}')$,
where the map $\varphi_K$ has been defined in Section~\ref{sec:basic-facts}.

\begin{theorem}\label{teo:firstInclusion} Let $n\ge 4$ and
$w \in \R^{n+1}$ be  in  a maximal cone of
$\mathcal{T}(Sev_n^2)$. Then $w$ is in a cone of type $I$, $II$ or
$III$.
\end{theorem}

\begin{proof}
Without loss of generality, we assume that
 $w \in {\rm Im}_{\val}$ and that there are one or two marked points in the subdivision
induced by the Newton Diagram of $w$.
If there are two marked points in different cells of the
subdivision, then $w$ is in a cone of type $I$. In the same way,
if there are two marked points in only one marked cell of the
subdivision, by Proposition~\ref{prop:exceptionalConfig}, we know
the marked cell is not the translation of an exceptional
configuration, hence $w$ is in a cone of type $II$.

Let us consider now the remaining case of only one marked cell
$\sigma=\{i_1,i_2,i_3\}$ with one marked point. There is a
polynomial $f \in Sev_n^2\cap \K[x]_w$ with $a,b\in \K^*$ double
roots. Then, because of item iii) in
Proposition~\ref{prop:ResultadosAliciaLuis}, $\val(a)=\val(b)$. To
see that $w$ is in a cone of type $III$,  we assume  without loss
of generality that $a=1$ and $b=\beta + ht^v+*$, with $v>0$. Using
Proposition~\ref{prop:unityRoot} we know that $\beta$ is a root of
unity of order $d$, where $d>1$ divides $\gcd(i_2-i_1,i_3-i_1)$.
Also, we assume without loss of generality that $w \in
\val(\mathcal V)$. For any index $k \notin \sigma$, we have by
Lemma~\ref{lemma:mult_of_0,1} that $D_{\sigma
  \cup \{k\}}(\beta)=0$. Since we assume that $D_{\sigma \cup \{k\}}(b)\neq 0$,
it follows that $h \neq 0$.

As $a=1$, we can assume $w_{i_1}=w_{i_2}=w_{i_3}=0$. We want to
prove that min$_{j \in J}\{w_j\}$ is attained at least twice, where $J=\{j \in \cA \ : \ d \nmid j-i_1\}$ .
Suppose that the minimum is only attained once at $i_4 \in J$, (with
$w_{i_4}>0$), and let $S=\{j\in \mathcal{A}\ : \ d \mid j-i_1\}$.
Because of Lemma~\ref{lemma:valD_I}
$\val(D_{\sigma\cup\{j\}}(b))=v$ for all $ j \in J$, hence ${\rm
min}_{j \in J}\{\val(D_{\sigma\cup\{j\}}(b))+w_j\}$ is also only
attained once at $i_4$. Then, by
Theorem~\ref{teo:matrix-by-minors} there exists at least one $s_0
\in S$ such that, for all $ k \in \cA$,
$\val(D_{\sigma\cup\{s_0\}}(b)) +w_{s_0}\le
\val(D_{\sigma\cup\{k\}}(b))+w_k$.

Consider the set $I=\{i_1,i_2,s_0\}$. Again by Theorem
~\ref{teo:matrix-by-minors} the minimum in the set
$\{\val(D_{I\cup\{k\}}(b))+w_k \ : \ k \in \cA\}$ is attained at
least twice. However, again by Lemma~\ref{lemma:valD_I},
$\val(D_{I\cup\{s\}}(b))=4v=\val(D_{\sigma\cup\{s_0\}}(b))$ when
$s\in S\!\setminus\! I$ and
$\val(D_{I\cup\{j\}}(b))=v=\val(D_{\sigma\cup\{j\}}(b))$ when
$j\in J$. Then {\small
$$\val(D_{I\cup\{k\}}(b))+w_k=\begin{cases}
\val(D_{\sigma\cup\{s_0\}}(b))+w_{i_3}=
\val(D_{\sigma\cup\{s_0\}}(b))=:D_{i_3} \qquad \quad \ \ \mbox{ if } k = i_3 \\
\val(D_{\sigma\cup\{s_0\}}(b))+w_k> D_{i_3}
\hspace{2.6cm} \mbox{ if } k \in S\!\setminus\!(I\!\cup\!\{i_3\}) \\
\val(D_{\sigma\cup\{k\}}(b))+w_k\ge
\val(D_{\sigma\cup\{s_0\}}(b))+w_{s_0}> D_{i_3} \qquad \mbox{ if }
k \in J \\ \infty> D_{i_3} \hspace{6.75cm} \mbox{ if } k \in I
\end{cases}$$} so the minimum would only be attained at $k=i_3$,
which is a contradiction.
\end{proof}

\section{Sufficient conditions}\label{sec:sufficient}

In this section we prove that each of the cones of type $I$, $II$
and $III$ belongs to the tropical Severi variety. Our proof is
constructive, that is, it is done showing that, for every $w$ in one
of the cones, we can compute a polynomial $f\in \K[x]_w\cap
Sev_n^2$.

The key is to choose an appropriate configuration $J=\{i_1, i_2,
i_3, i_4\}$ and study the matrix $\mathbf{M}_J(x)$ as well as the
valuation of its determinant $D_J=\det(\mathbf{M}_J)$, where the
matrix $\M$ is as in Definition~\ref{def:matrixM(x)}. We will need
many technical lemmas, that are collected in
Appendix~\ref{app:PropertiesM}. In particular, we have that  $D_J
\in\Z[x]-\{0\}$ for all 4-index set $J$, by Lemma~\ref{lemma:B3}.

We want to prove that there exists $c=(c_0, \dots, c_n)$ with
$\val(c)=w$ and $b$ different from 1 that solve the system
$\mathbf{M}(b) \, c^T=0$, and therefore $w$ is a point on
$\mathcal{T}(Sev_n^2)$. Moreover, we are going to find $f\in
\K[x]_w$ with multiple roots $1$ and $b$ such that $c_i=t^{w_i}$
for all $ i \not\in J$. Then, we are looking for an appropriate
$b$ and $(c_{i_1}, c_{i_2}, c_{i_3}, c_{i_4})$ with each of its
coordinates of the correct valuation that solve the equation
system:
\begin{equation} \label{eq:typeIIv2}
\mathbf{M}_J(b) \cdot
(c_{i_1},c_{i_2},c_{i_3},c_{i_4})^T=-(\sum\limits_{i \not\in
J}t^{w_i}, \sum\limits_{i \not\in J}it^{w_i}, \sum\limits_{i
\not\in J}b^it^{w_i}, \sum\limits_{i \not\in J}ib^it^{w_i})^T
\end{equation}

The easiest case are cones of type $I$.

\subsection{Type $I$}

Consider $w \in {\rm Im}_{\val}$ a point in the interior of a cone
of type $I$. That is, the subdivision induced by $w$ has two
different marked cells with one marked point each. Without loss of
generality, we can assume that the slopes of the lower facets from
the Newton diagram corresponding to those marked cells are 0 and
$-v$ where $v> 0$. Then, the Newton diagram is as in
Figure~\ref{fig:Type I}.

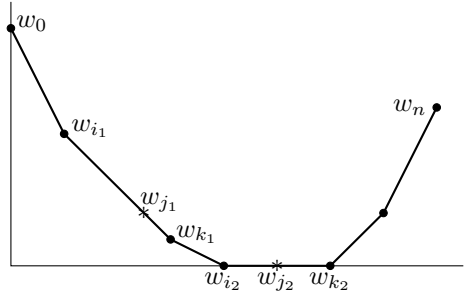
\begin{figure}[hbtp]
\begin{tikzpicture}[scale=0.7]
\draw (0,5) -- (0,0) -- (8.5,0); \draw[fill] (0,4.5) circle (2pt)
(1,2.5) circle (2pt) (3,0.5) circle (2pt); \draw[fill] (4,0)
circle (2pt) (6,0) circle (2pt) (7,1) circle (2pt) (8,3) circle
(2pt); \draw[thick] (0,4.5) -- (1,2.5) -- (3,0.5) -- (4,0) --
(6,0) -- (7,1) -- (8,3); \draw (2.5,1) node {$*$} (5,0) node
{$*$}; \draw (0.4,4.6) node {$w_0$} (1.5,2.6) node {$w_{i_1}$}
(2.8,1.2) node {$w_{j_1}$}; \draw (3.5,0.6) node {$w_{k_1}$};
\draw (4,-0.3) node {$w_{i_2}$} (5,-0.3) node {$w_{j_2}$} (6,-0.3)
node {$w_{k_2}$}; \draw (7.5,3) node {$w_{n}$};
\end{tikzpicture}
\caption{Two marked cells with different slopes}\label{fig:Type I}
\end{figure}

That is, there are two marked cells, $\{i_1, j_1, k_1\}$ and
$\{i_2, j_2, k_2\}$ where $k_1 \le i_2$, in the corresponding
subdivision. We can assume $w_{i_2}=w_{j_2}=w_{k_2}=0$. Let
$-v=\frac{w_{k_1}-w_{i_1}}{k_1-i_1}$ be the slope of the
non-horizontal marked segment.

\begin{theorem} \label{teo:Imain}
Let $n \ge 4$ and $w\in {\rm Im}_{\val}$ be a point in a cone of
type $I$. Then $w$ is an element of $\mathcal{T}(Sev_n^2)$.
\end{theorem}
\begin{proof}
Using the same notation and assumptions as before, we are going to
present a polynomial $f$ complying with the request. In
particular, we can assume $w$ is a  point in the interior of the
cone. It is worth mentioning that this is only one choice but
there are infinite polynomials that satisfy the theorem (see
\cite{Payne}). {Moreover, if $w \in \Z^{n+1}$, then our
proposed solution $f$ lies in $\C(t)[x]$.}

We will build a polynomial $f=\sum\limits_{i=0}^nc_ix^i\in
Sev_n^2$ such that $f\in \K[x]_w$ and its double roots are $a=1$
and $b=t^v$. Let $J=\{i_1, j_1, j_2, k_2\} $ and take
$c_l=t^{w_l}$ for all index $l \not\in J$.

We want to see that there are $c_{j_2}, c_{k_2}$ of valuation 0
and $c_{i_1}, c_{j_1}$ of valuations $w_{i_1}$ and $w_{j_1}$
respectively that are solutions of the system in Equation
\eqref{eq:typeIIv2}. Let us see that there exist $\gamma_{i_1},
\gamma_{i_2}$ such that $c_{i_1}=\gamma_{i_1}t^{w_{i_1}}+*$,
$c_{j_1}=\gamma_{j_1}t^{w_{j_1}}+*$ and $\gamma_{j_2},
\gamma_{k_2}$ with $c_{j_2}=\gamma_{j_2}+*$ and
$c_{k_2}=\gamma_{j_1}+*$. {Replacing this notation in Equation
\eqref{eq:typeIIv2}, we get the system
$$ \begin{pmatrix} t^{w_{i_1}} & t^{w_{j_1}} & 1 & 1 \\
i_1t^{w_{i_1}} & j_1t^{w_{j_1}} & j_2 & k_2 \\
t^{w_{i_1}+vi_1} & t^{w_{j_1}+vj_1} & t^{vj_2} & t^{vk_2} \\
i_1t^{w_{i_1}+vi_1} & j_1t^{w_{j_1}+vj_1} & j_2t^{vj_2} & k_2t^{vk_2}
\end{pmatrix} \begin{pmatrix} \gamma_{i_1} \\
\gamma_{j_1} \\ \gamma_{j_2} \\ \gamma_{k_2}
\end{pmatrix} = \begin{pmatrix} -1+* \\
-i_2+* \\ -t^{w_{k_1}+vk_1}+*\\
-k_1t^{w_{k_1}+vk_1}+* \end{pmatrix}
$$}
Dividing the third and fourth rows of the $4\times 4$ matrix by $t^{vk_1+w_{k_1}}$
we obtain a system where every term has nonnegative valuation
because $i_1,j_1,k_1$ are the indices where the minimum is
attained in the segment with slope $-v$. {Moreover, the valuation of the entries in those rows
associated to $j_2$ and $k_2$ is positive because of the convexity of the Newton Diagram.}

Setting $t=0$ we obtain the $4 \times 4 $ invertible system over $K$ (in matrix form):
$$\begin{pmatrix} 0 & 0 & 1 & 1 \\ 0 & 0 & j_2 & k_2 \\
1 & 1 & 0 & 0 \\ i_1 & j_1 & 0 & 0 \end{pmatrix}
\begin{pmatrix}\gamma_{i_1} \\
\gamma_{j_1} \\ \gamma_{j_2} \\ \gamma_{k_2}
\end{pmatrix} = \begin{pmatrix} {-}1 \\ {-}i_2 \\ {-}1 \\ {-}k_1
\end{pmatrix}.$$
 Hence,  system~\eqref{eq:typeIIv2} has
also a unique solution, which equals:
$$(c_{i_1}, c_{j_1}, c_{j_2}, c_{k_2})=\left(
{-}\frac{j_1-k_1}{j_1-i_1}t^{w_{i_1}}+*,
{-}\frac{k_1-i_1}{j_1-i_1}t^{w_{j_1}}+*,
{-}\frac{k_2-i_2}{k_2-j_2}+*,
{-}\frac{i_2-j_2}{k_2-j_2}+*\right),$$
and has the correct valuations.
\end{proof}

\subsection{Type $II$}

Type $II$ cones are the most difficult case to prove completely.
We follow the same approach as for type $I$ cones, but we have to
separate the case where the marked cell is an affine image of an
exceptional configuration. The choice of the node $b$ is more
complicated and subtle than in the case of nodes of type $I$.

\subsubsection{The general case}

For points $w$ in cones of type $II$ and $III$ we will assume that
the marked cell has slope $0$ and we will build a polynomial $f
\in \K[x]_w$ with double roots $1$ and $b$ of valuation $0$. We
need that $b=\beta+*$ and $D_J(b)\neq 0$, but this is possible
since $D_J\ne 0$ (see Lemma \ref{lemma:B3}). Note that, as
$\val(b)=0$, $\beta\neq 0$. Also, by
Corollary~\ref{corol:differentResidue} we can assume $\beta\neq
1$. Multiplying by the adjoint matrix adj$(\mathbf{M}_J(b))$ of
$\mathbf{M}_J(b)$, the solutions for all $ 1 \le j \le 4$ of
Equation (\ref{eq:typeIIv2}) are
\begin{equation}\label{eq:typeIIv3}
c_{i_j}= -\sum\limits_{i \not\in J}
\pm\frac{D_{(J\cup\{i\})-\{i_j\}}(b)}{D_J(b)} t^{w_{i}},
\end{equation}
where the signs depend on the number of permutations necessary to
reorder from lowest to highest the ordered set $J$ as we replace
$i_j$ by $i$ {(see Notation \ref{not:M_J})}. As $\val(b)=0$, the
entries of the matrix adj$(\mathbf{M}_I(b))$ are of valuation at
least $0$. Also, if $\val(w_i)>0$ for every $i \not\in J$, then
every coordinate in the right hand part of the system has positive
valuation. Therefore, if the solution
$(c_{i_1},c_{i_2},c_{i_3},c_{i_4})$ fulfills $\val(c_{i_j})=0$ for
some $ 1 \le j \le 4$, the valuation of $D_J(b)$ has to be
positive. Hence, $D_J(\beta)=0$.

\smallskip

Consider now a point $w$ in the interior of a cone of type $II$.
Then there is one marked cell $J=\{i_1,i_2,i_3,i_4\}$ with two
marked points. We can assume $w_{i_1}=w_{i_2}= w_{i_3}=
w_{i_4}=0$, and $w_j>0$ for all $j \not\in J$, so we are in the
situation stated above.

\begin{remark} \label{rem:exTypeIIminors} Let
$f=(x-1)^2(x-b)^2(x-a) \in \K[x]$ where $b=\beta+*$ with $\beta
\in K\setminus\{0,1\}$ and $\val(a)\ne 0$. From item i) in
Proposition~\ref{prop:ResultadosAliciaLuis}, the subdivision
induced by the vector of valuations of the coefficients of $f$ has
two cells, one of lattice length $1$ and another one of lattice
length $4$. Then, for $w$ to be a point in the interior of a cone
of type $II$, the cell of lattice length $4$ has to be either
$\{0,1,3,4\}$ or $\{1,2,4,5\}$ by
Proposition~\ref{prop:exceptionalConfig}. It is easy to see that
the only zeros in $K\setminus\{0,1\}$ of either $D_{\{0,1,3,4\}}$
or $D_{\{1,2,4,5\}}$ are $-2\pm\sqrt{3}$, as used in
Example~\ref{ex:typeII}.
\end{remark}

In what follows, the expression that a value is ``repeated at most once'' means that it
can occur $1$ or $2$ times (but not more).

\begin{theorem}\label{teo:IImaina}
Let $ w \in \R^{n+1}$ be a point a cone of type $II$ with marked
cell $J=\{i_1,i_2,i_3,i_4\}$. If there is a $\beta \in K\setminus\{0,1\}$
root of $D_J$ such that each of the powers
$\{\beta^{i_1},\beta^{i_2},\beta^{i_3},\beta^{i_4}\}$ {is repeated
at most once}, then $w$ is in $\mathcal{T}(Sev_n^2)$.
\end{theorem}
\begin{proof}
Without loss of generality, we can assume $w \in {\rm Im}_{\val}$ is a
generic point in the interior of the cone, and $w_{i_j}=0$ for all
$ 1 \le j \le 4$. By Lemma~\ref{lemma:goodsecondmonomial} the set
$S=\{i \in \cA-J \ : \ D_{(J\cup\{i\})-\{i_1\}}(\beta)\ne 0\}\ne
\emptyset$. As $w$ is a generic point, we can assume that there
exists $i_5\in S$ such that $0<w_{i_5}<w_i$ for all $ i \in
S-\{i_5\}$. Since $D_J(\beta)=0$, by
Lemma~\ref{lemma:alternativa}, $
D_{(J\cup\{i_5\})-\{i_j\}}(\beta)\ne 0$ for all $ 1 \le j \le 4$.
Note that, since $\{\beta^{i_j}\}_{j=1}^4$ are not all equal, by
Lemma~\ref{lemma:triple_root} the multiplicity of $\beta$ as a
root of $D_J$ is at most two. Then, by
Lemma~\ref{lemma:double_root} the multiplicity $m$ of $\beta$ as a
root of $D_J$ is lower or equal to the multiplicity $m_{ej}$ of
$\beta $ as a root of $D_{(J\cup\{e\})-\{i_j\}}$ for all $ e
\not\in S\cup J$ and $ 1 \le j \le 4$. Let $b= \beta
+ht^{\frac{w_{i_5}}{m}}\in \K$ such that $ D_J(b)\ne 0$ and
$f=\sum\limits_{i=0}^nc_ix^i$ such that $c_i=t^{w_i}$ for all $i
\not\in J$. We are going to see that the coefficients
$\{c_{i_j}\}_{j=1}^4$ such that $1,b$ are multiple roots of $f$
fulfill $\val(c_{i_j})=0$ for all $ 1 \le j \le 4$ and therefore
$w$ is in $\mathcal{T}(Sev_n^2)$. Note that $\val(D_J(b))=w_{i_5}$
and, for all $ e \not\in S\cup J$ and $ 1 \le j \le 4$,
$\val(D_{(J\cup\{e\})-\{i_j\}}(b))=m_{ej}\frac{w_{i_5}}{m}\ge
w_{i_5}.$ We saw in Equation~(\ref{eq:typeIIv3}) that the
solutions $(c_{i_1},c_{i_2},c_{i_3},c_{i_4})$ are of the form
$c_{i_j}=-\sum\limits_{i \not\in
J}\pm\dfrac{D_{(J\cup\{i\})-\{i_j\}}(b)}{D_J(b)}t^{w_i}$ for all $
1 \le j \le 4$. If we analyze the valuation for every term of that
sum, we can see that:
\begin{itemize}
\item If $i=i_5$, then
$\val\left(\dfrac{D_{(J\cup\{i_5\})-\{i_j\}}(b)}{D_J(b)}c_{i_5}\right)=0-
w_{i_5}+w_{i_5}=0.$
\item If $i \not\in S\cup J$, then
$\val\left(\dfrac{D_{(J\cup\{i\})-\{i_j\}}(b)}{D_J(b)}c_i\right)\ge
w_{i_5}-w_{i_5}+w_i>0$ because $w_i>0$.
\item If $i \in S$, then
$\val\left(\dfrac{D_{(J\cup\{i\})-\{i_j\}}(b)}{D_J(b)}c_i\right)=
0- w_{i_5}+w_{i}>0$ because $w_i>w_{i_5}$.
\end{itemize}
Then,
$$c_{i_j}=\pm\frac{D_{(J\cup\{i_5\})-\{i_j\}}(b)}{D_J(b)}t^{w_{i_5}}+* $$
and has valuation 0 for all $ 1 \le j \le 4$.
\end{proof}

\subsubsection{Exceptional Configurations} \label{sec:excepConfig}

In Theorem \ref{teo:excepConfig} we proved that the sets
$\sigma=\{i_1,i_2,i_3,i_4\}$ such that every root $\beta\in
K\setminus\{0,1\}$ of $D_\sigma(x)$ satisfies that there are at least
three of the powers
$\{\beta^{i_1},\beta^{i_2},\beta^{i_3},\beta^{i_4}\}$ equal, are
precisely those of the form $\sigma = (J(\times
s))(+r)=\{sj_1+r,sj_2+r, sj_3+r, sj_4+r\}$ where $s \in \N $, $r
\in \Z_{\ge 0}$ and $J=\{j_1,j_2,j_3,j_4\}$ is an exceptional
configuration (see Definition~\ref{def:exceptional
configurations}).

In Proposition~\ref{prop:exceptionalConfig} we already proved that
if $w$ is a point such that the subdivision induced by $w$ has
only one marked cell $J$ and this cell is a translation of an
exceptional configuration, then $w$ is not in
$\mathcal{T}(Sev_n^2)$. In the next theorem we will see that if
the marked cell $\sigma$ is a image under an affine map $(J(\times
s))(+r)$ of an exceptional configuration such that
$s=gcd(i_4-i_1,i_3-i_1,i_2-i_1)>1$, then $w$ is in fact in
$\mathcal{T}(Sev_n^2)$. With this result we complete the proof
that all $w$ in the interior of a type $II$ cone is in
$\mathcal{T}(Sev_n^2)$.

\begin{theorem} \label{teo:IImainb}
Let $J=\{j_1,j_2,j_3,j_4\}$ be an exceptional configuration. Let
$\Pi$ be a subdivision of $\cA$ such that the only marked cell is
$I=\{i_1,i_2,i_3,i_4\}$ where $I=(J(\times s))(+r)$ for $s\in \N$
and $r \in \Z_{\ge 0}$. Let $w\in {\rm
Im}_{\val}\subseteq\R^{n+1}$ be a point in the cone $C_\Pi$
defined by the closure of the set of points in $\R^{n+1}$ whose
Newton diagram induces the subdivision $\Pi$. If $s>1$, then $w$
is a point in $\mathcal{T}(Sev_n^2)$.
\end{theorem}
\begin{proof}
We can assume that $w\in {\rm Im}_{\val}$ is a generic point in
the interior of the cone and that
$w_{i_1}=w_{i_2}=w_{i_3}=w_{i_4}=0$. We may also assume without
loss of generality that $i_5$ is the only index where $\min\{w_i \
: \ i \not\in I \mbox{ and } s \nmid i\}$ is attained. Let
$\beta\in K\setminus\{0,1\}$ be a primitive $s$-th root of unity and
$f=\sum\limits_{i=0}^nc_ix^i \in \K[x]$ such that $c_i=t^{w_i}$
for all $i \not\in I$. We want to find $\{c_i\}_{i \in I}\subseteq
\K$ such that $\val(c_i)=0$ for all $ i \in I$ and $f$ has two
multiple roots $1$ and $ b= \beta+*$. Take $b=\beta + ht^v$, such
that $h \in K^*$, $v=\frac{w_{i_5}}{3}$ and $D_I(b)\ne 0$.
Then the system in Equation \ref{eq:typeIIv2} has a unique
solution in $\K^4$ (presented in Equation~(\ref{eq:typeIIv3})).

By Lemma~\ref{lemma:valD_I} we know that $\val(D_I(b))=4v$,
and that for every $i \not\in I$
$\val(D_{I\cup\{i\}-\{i_j\}}(b))=v$ if $s \mid i $ and
$\val(D_{I\cup\{i\}-\{i_j\}}(b))=4v$ if $s \nmid i $. Then,
\begin{itemize}
\item $\val(\frac{D_{(I\cup\{i\})-\{i_j\}}(b)}{D_I(b)}
t^{w_{i}})=v-4v+w_i=w_i-3v \mbox{ if } s \nmid i.$
\item $\val(\frac{D_{(I\cup\{i\})-\{i_j\}}(b)}{D_I(b)} t^{w_{i}}) =
4v-4v+w_i=w_i \mbox{ if } s \mid i.$
\end{itemize}
Since $ v = \frac{w_{i_5}}{3}$, we have that:
\begin{itemize}
\item $\val(\frac{D_{(I\cup\{i_5\})-\{i_j\}}(b)}{D_I(b)}
t^{w_{i_5}})= w_{i_5}-3v=w_{i_5}-3\frac{w_{i_5}}{3}=0.$
\item $\val(\frac{D_{(I\cup\{i\})-\{i_j\}}(b)}{D_I(b)} t^{w_{i}})=
w_i-3v=w_i-w_{i_5}>0 \mbox{ if } s \nmid i$ because of the choice
of $i_5$.
\item $\val(\frac{D_{(I\cup\{i\})-\{i_j\}}(b)}{D_I(b)} t^{w_{i}})=
w_i>0 \mbox{ if } s \mid i$.
\end{itemize}
Therefore, we have that $\min_{i \not\in
I}\{\val(\frac{D_{(I\cup\{i\})-\{i_j\}}(b)}{D_I(b)} t^{w_{i}})\}$
is attained only at $i_5$ where is zero, hence $\val(c_{i_j})=0$
for all $ 1 \le j \le 4$.
\end{proof}

\subsection{Type $III$}
Finally, consider $w \in \R^{n+1}$ a point in the interior of a
cone of type $III$. That is, $w$ is such that the induced
subdivision has only one marked cell $\{i_1, i_2, i_3\}$ with one
marked point $i_2$, and a hidden tie $\{i_4,i_5\}$. Assume also,
without loss of generality, that $w_{i_j}=0$ for all $1 \le j \le
3$ and $w_k>0$ for all $ k \not\in I$. Then:
\begin{itemize}
\item $\gcd(i_2-i_1,i_3-i_1)>1$.
\item There is an integer $d>1$, which divides of $\gcd(i_2-i_1,i_3-i_1)$,
such that the set $\{w_j \ : \ j \in \cA \mbox{ and } j \not\equiv
i_1 \mod d\}$ attains its minimum twice at $\{i_4,i_5\}$.
\end{itemize}

\begin{theorem}\label{teo:SufficiencyIII}
Let $w\in \R^{n+1}$ be a point in a cone of type $III$. Then $w$
is an element of $\mathcal{T}(Sev_n^2)$.
\end{theorem}
\begin{proof}
Without loss of generality, we assume $w\in {\rm Im}_{\val}$ a
generic point in the interior of the cone and
$w_{i_1}=w_{i_2}=w_{i_3}=0$ where $\{i_1, \dots, i_5,d\}$ are as
in the notation above. Let $\beta$ be a primitive root of unity of
order $d$ and $b=\beta+t^{w_{i_4}}$.
Note that we can write $b^{dn}=1+dn\beta^{-1}t^{w_{i_4}}+*$ for all $ n \in \N$.

Let $J=\{i_1, i_2,i_3,i_4\}$. Using Lemma~\ref{lemma:valD_I}, $\val(D_J(b))=w_{i_4}$.

To prove that there exists $f=\sum\limits_{i=0}^nc_ix^i$ with
double roots $1$ and $b$ in $ Sev_n^2\cap \K[x]_w$, we need to {see}
that there exist a vector of coefficients $(c_i)_{i \in \cA}$ that
satisfies Equation (\ref{eq:typeIIv2}). As the minor $D_J(b)$ is
not zero, there exists a unique vector of solutions $c=(c_{i_1},
c_{i_2},c_{i_3},c_{i_4}) \in \K^4$ obtained by taking $c_i=t^{w_i}$
for all $ i \not\in J$. It only remains to be seen that
$\val(c_{i_j})=0$ for all $ 1 \le j \le 3$ and
$\val(c_{i_4})=w_{i_4}$.

Let $S$ be the set of all $ s \in \cA$ such that $s\equiv i_1 \
\mod d$ and $0<\val(w_s)\le w_{i_4}$. {In particular, because of the second defining condition on $S$ we have $i_1,i_2,i_3 \not\in S$.} As $w$ is generic, we can
assume that either $S=\emptyset$ or $\{s \in S \ : \ w_s \le w_k
\mbox{ for all } k \in S\}=\{s_0\}$. We re-write the system of
Equation (\ref{eq:typeIIv2}) as
\begin{equation}\label{eq:typeIII}
\mathbf{M}_J(b) \cdot c^T=-\mathbf{M}_{S}(b)\cdot (t^{w_s})_{s \in
S}^T - \mathbf{M}_{\cA-(S\cup J)}(b)(t^{w_j})^T_{j \not\in S\cup
I},
\end{equation}
we are going to find the valuation of the solutions
$c^{(1)}=(c_{i_j}^{(1)})_{j=1}^4$ and
$c^{(2)}=(c_{i_j}^{(2)})_{j=1}^4$ for the following associate
systems:
\begin{equation}\label{eq:primera_parte}
\mathbf{M}_J(b) \cdot (c^{(1)})^T= -\mathbf{M}_{S}(b)\cdot
(t^{w_s})_{s \in S}^T =-(
t^{w_{s_0}},s_0t^{w_{s_0}},b^{s_0}t^{w_{s_0}},s_0b^{s_0}t^{w_{s_0}})^T,\end{equation}

\begin{equation}\label{eq:segunda_parte}\begin{split}
\mathbf{M}_J(b) \cdot (c^{(2)})^T &= -\mathbf{M}_{\cA-(S\cup
J)}(b)\cdot (t^{w_j})_{j \not\in S\cup I}^T \\ &=-(t^{w_{i_5}}+
*,i_5t^{w_{i_5}}+*,b^{i_5}t^{w_{i_5}}+*,i_5b^{i_5}t^{w_{i_5}}+*)^T.
\end{split}
\end{equation}
We denote by $\det(\M_{J'}(b) |\, C)$
the determinant of the matrix in $\K^{4 \times 4}$ where the first
columns are given by the matrix $\mathbf{M}_{J'}(b)$ (for some set
$J'$ {of cardinal 3}) and the last one is a fixed column $C$. In a similar way $\det(C\,|\, \M_{J'}(b))$ is the determinant of the matrix
$\mathbf{M}_{J'}(b)$ extended with a first column $C$.

First, if $S\neq 0$, we solve Equation~(\ref{eq:primera_parte}).
Consider the case of $c_{i_1}^{(1)}$ (as $s_0\equiv i_1 \mod d$ we
can use Lemma~\ref{lemma:valD_I} and we have, for $C=
(t^{w_{s_0}},s_0t^{w_{s_0}},b^{s_0}t^{w_{s_0}},s_0b^{s_0}t^{w_{s_0}})
$, that {\small$$\det(C\,|\, \mathbf{M}_{\{i_2,i_3,i_4\}}(b))=
{t^{w_{s_0}}\det(\mathbf{M}_{\{s_0,i_2,i_3,i_4\}}(b))}
$$} has valuation $w_{s_0}+w_{i_4}$. Therefore, by Cramer's rule
$$\val(c_{i_1}^{(1)})=w_{s_0}+w_{i_4}-w_{i_4}=w_{s_0}>0.$$

As $i_1\equiv i_2\equiv i_3 \mod d,$ we can prove equivalently
that $\val(c_{i_2}^{(1)})=w_{s_0}>0$ and
$\val(c_{i_3}^{(1)})=w_{s_0}>0$. The case of $c_{i_4}$ is slightly
different because we are replacing the fourth column and
$s_0\equiv i_j \mod d$ for all $ 1 \le j \le 3$. Using
again Lemma~\ref{lemma:valD_I} and
$C=(t^{w_{s_0}},s_0t^{w_{s_0}},b^{s_0}t^{w_{s_0}},s_0b^{s_0}t^{w_{s_0}}
)$, the valuation
$\val(\det(\mathbf{M}_{\{i_1,i_2,i_3\}}(b)|\,C))=
w_{s_0}+4w_{i_4}$, and therefore $\val(c_{i_4}^{(1)})\ge
w_{s_0}+4w_{i_4}-w_{i_4}> w_{i_4}$.

Consider now the case of $c_{i_1}^{(2)}$, we have to take into
account that $i_5 \not\equiv i_1 \mod d$. Then, considering the
vector
$C=(t^{w_{i_5}}+*,i_5t^{w_{i_5}}+*,b^{i_5}t^{w_{i_5}}+*,i_5b^{i_5}t^{w_{i_5}}+*)$,
$$\det(C\,|\,\mathbf{M}_{\{i_2,i_3,i_4\}}(b))=
t^{w_{i_5}}
\beta^{i_5+i_2}(i_3-i_2)(i_5-i_4)(\beta^{i_2-i_5}-1)(\beta^{i_4-i_2}-1)+*$$
has valuation $w_{i_5}=w_{i_4}$, and hence
$\val(c_{i_1}^{(2)})=0$. Again, as $i_1\equiv i_2 \equiv i_3 \mod
d$ we have $\val(c_{i_2}^{(2)})=0$ and $\val(c_{i_3}^{(2)})=0$.
For the case of $c_{i_4}^{(2)}$, {we can find the valuation similarly to
Lemma~\ref{lemma:valD_I}}. Taking
$C=(t^{w_{i_5}}+*,i_5t^{w_{i_5}}+*,b^{i_5}t^{w_{i_5}}+*,i_5b^{i_5}t^{w_{i_5}}+*)$,
we obtain $$\det(\mathbf{M}_{\{i_1,i_2,i_3\}}(b)|\,C)=
-t^{w_{i_5}+w_{i_4}}\beta^{2i_1-1}(i_3-i_2)(i_3-i_1)(i_2-i_1)(\beta^{i_5-i_1}-1)
+*.$$ Hence $\val(c_{i_4}^{(2)})=w_{i_5}+w_{i_4}-w_{i_4}=w_{i_4}$.

Taking
$c=(c_{i_1},c_{i_2},c_{i_3},c_{i_4})=(c_{i_1}^{(1)}+c_{i_1}^{(2)},
\dots, c_{i_4}^{(1)}+c_{i_4}^{(2)})$, this is the solution for
Equation~(\ref{eq:typeIII}). As $\val(c_{i_j}^{(1)})>0$ and
$\val(c_{i_j}^{(2)})=0$ for all $ 1 \le j \le 3$,
$\val(c_{i_4}^{(1)})>w_{i_4}$ and $\val(c_{i_4}^{(2)})=w_{i_4}$,
then every coordinate of $c$ satisfies $\val(c_{i_j})=w_{i_j}$ for
all $ 1 \le j \le 4$ as wanted.
\end{proof}

\appendix

\section{Proof of Lemma~\ref{lemma:minimal}}

We give here the proof of Lemma~\ref{lemma:minimal}, that asserts
that the tropicalization of a linear homogeneous ideal $I$ over a
valuated field $\K$ with infinite residue field is determined by the circuits.

\begin{proof} We denote by
$L(I)$ the intersection $\bigcap_{\ell \in I : \ell
\textrm{ linear form}}V(trop(\ell))$.
Let us first prove the equality
$\mathcal{T}(I) = L(I)$.
Clearly,
$\mathcal{T}(I)\subseteq L(I)$.
We will see the other inclusion by induction
on the codimension of $V=V(I)\cap(\K^*)^n$.

{Note that $V=\emptyset$ if and only if $V(I)$ is contained in a hyperplane $x_i=0$ and so $x_i\in I$ and $\mathcal{T}(I)=L(I)=\emptyset$. Hence, we can assume that $V\neq \emptyset$.}

If codim$(V)=1$, then
$V$ is a hyperplane and there exists a linear form $\ell$ such
that $V=V(\ell)\cap(\K^*)^n$. In this case is clear that
$\mathcal{T}(I)=V(trop(\ell)).$
If the result is true for every codimension less than or equal to
$r$, let us take $V$ of codimension $r+1$. Let $ p \in L(I)$.
If there exists $e_i$ in the canonical basis
such that $e_i \in V(I)$ for some $i \in \{1, \dots, n\}$, then
for all $\ell \in I$ linear form,
$\ell(x)=\sum_{j=1}^n\ell_jx_j$ with $\ell_i=0$. Let
$\pi_i: \K^n \longrightarrow \K^{n-1}$ be the projection
$\pi_i(x_1, \dots, x_n)=(x_1, \dots, x_{i-1},x_{i+1}, \dots,
x_n)$. Then, $p \in \mathcal{T}(I)$ if and only if $\pi_i(p) \in
\mathcal{T}(\pi_i(V))$ and the result follows by induction in the
codimension.

If $e_i \not\in V(I)$ for all $ 1 \le i \le n$, $V(I)\subsetneq
V(I)\oplus\langle e_i \rangle$ for all $i$.
By our inductive hypothesis, $\mathcal{T}(V(I)\oplus\langle e_i\rangle)$
equals the intersection of $V(trop(\ell))$, for all linear forms $\ell$
with $ V(I)\oplus\langle
e_i\rangle \subseteq \{\ell=0\}$.
 Therefore, $ p
\in \bigcap_{i=1}^n\mathcal{T}(V(I)\oplus\langle
e_i\rangle)$. Without loss of generality we can assume $p=(0,
\dots, 0)$. Hence, for each $1 \le i \le n$ there is a point $u_i
= (u_{i1},\dots , u_{in})\in V(I)$ such that for all $i \neq j$,
$\val(u_{ij}) = 0$. As codim$(V)>1$, the matrix $U$ with
$U_{ij}=u_{ij}$ for all $ 1 \le i,j \le n$ has rank at most
$n-2$. Then, $\det(U)=0$, and so there exist $i \in \{1, \dots, n\}$ such
that $\val(u_{ii})\ge 0$. Since the minor {obtained
from the matrix $U$ deleting the $i$-th row and column} is also
singular, there must be another element $j \in \{0, \dots,
n\}-\{i\}$ with $\val(u_{jj})\ge 0$. Without loss of
generality, let $\{i,j\}=\{1,2\}$. If $\val(u_{11})=0$ or
$\val(u_{22})=0$, then $p=\val (u_1) \in \mathcal{T}(I)$ or
$p=\val (u_2) \in \mathcal{T}(I)$. If not, we can take a generic
$\lambda\in\K^*$ of valuation zero such that $trop(u_1+\lambda
u_2)=p$ and therefore $p \in \mathcal{T}(I)$. In fact, let
$\lambda$ be generic so $\val (u_1+\lambda u_2)_i= 0$ for all $3
\le i\le n$. As $\val (u_{11})>0$ and $\val (\lambda u_{21})=0$,
then $\val (u_{11}+\lambda u_{21})=0$. Similarly $\val
(u_{12}+\lambda u_{22})=0$.

To see that it is enough to consider the circuits in $I$ to describe
$\mathcal{T}(I)$, recall
that, for each minimal support, there is only (up to multiplication
by a constant) one linear form in $I$. Let us call $\{\ell'_1, \dots,
\ell'_r\}$ linear forms generating $I$. Clearly,
$\mathcal{T}(I)\subseteq \cap_{i=1}^r V(trop(\ell'_i)).$
Given $p \not\in \mathcal{T}(I)$, there is a
linear form $\ell \in I$ such that $p \not\in V(trop(\ell))$. It
suffices to check that we can take $\ell$ with minimal support.
Let $\ell_1:=\ell$. If $\ell_1$ does not have minimal support,
then there exists a linear form $g_1 \in I$ whose support is
contained in the support of $\ell_1$. If $p \not\in V(trop(g_1))$,
we are done. If $ p \in V(trop(g_1))$, let $x_{m_1}$ and $x_{m_2}$
be different variables such that the associated linear forms in $trop(\ell_1)$
and $trop(g_1)$ attain
respectively the minimum at $p$. We can take $\ell_2$ a linear
combination of $\ell_1$ and $g_1$ to make zero the monomial
$x_{m_2}$ in $\ell_2$ without modifying where $trop(\ell_2)$ attains its
minimum, and proceed recursively.
\end{proof}

\section{The matrix $\mathbf{M}$ and its minors}\label{app:PropertiesM}

In this section, we prove all the claims about the matrix $\M$ in
Definition~\ref{def:matrixM(x)}, its roots, rank etc. used in the
paper. First of all, by the following lemma we can assume (when
needed, to simplify the notation) that $i_1=0$.

\begin{lemma} \label{lemma:i=0} With the previous notation, the minor $ D_J $ is
invariant by translations of $i_1,i_2, i_3$ and $i_4$ (up to a
monomial factor). More specifically, for $s
\in \N$
\[x^{2s}D_J(x)=D_{J(+s)}(x) \mbox{ and } D_{J(\times s)}(x)=s^2\cdot D_J(x^s)\]
Also, given positive $i < j< k$,
 {\small$$
 D_{\{0,i,j,k\}}(x)
=i k \, x^{i}(x^{k-i}-1)(x^{j}-1)
- i j \, x^{i}(x^{j-i}-1)(x^{k}-1)-
j k \, x^{j}(x^{k-j}-1)(x^{i}-1),$$}
and we have the equality
\[ D_{\{0,i,j,k\}}(x) = x^{i+j+k} \, D_{\{0,i,j,k\}}(x^{-1}).\]
\end{lemma}
\begin{proof} It is a straightforward computation from the properties
of the determinant.
\end{proof}

\begin{lemma} \label{lemma:B3}
Let $J=\{i_1, i_2, i_3, i_4\}$ be a subset of $\cA$ with
$i_1<i_2<i_3<i_4$. Then $D_J$ is not the zero polynomial, its
degree (as a polynomial in $x$) is $i_3+i_4$ and its order at the
origin equals $i_1+i_2$. Moreover, the matrix $\mathbf{M}(x)$ has
rank 4 if $x\neq 0,1$ and $n\ge 3$.
 \end{lemma}

\begin{proof}
By Lemma~\ref{lemma:i=0},
\[D_J(x) = x^{2 i_1} \, D_{\{0, i_2-i_1,i_3-i_1,i_4-i_1\}} = (i_2-i_1)(i_4-i_3) \left(x^{i_3+i_4} + \dots +
x^{i_1+i_2}\right).\] If we consider the submatrix $\mathbf{M}_J
\in (\Z[x])^{4\times4}$ given by $J=\{0,1,2,3\}$, it holds that
$D_J=x(x-1)^4$, and hence rank$(\mathbf{M}(x))=4$ for all $x\neq
0,1$.
\end{proof}

\begin{lemma} \label{lemma:rank2}
Let $J=\{i_1, i_2, i_3, i_4\}$ be a subset of $\cA$ such that
$i_1<i_2<i_3<i_4$. Let $\beta \in K$. Then, the matrix
$\mathbf{M}_J(\beta)$ has rank 2 if and only if
$\beta^{i_1}=\beta^{i_2}=\beta^{i_3}=\beta^{i_4}$.
\end{lemma}
\begin{proof}We can assume $\beta\ne 0$ since the result is trivial otherwise.
It is clear that $\beta^{i_1}=\beta^{i_2}=\beta^{i_3}=\beta^{i_4}$
is a sufficient condition for rank$(\mathbf{M}_J(\beta))=2$. To
see that it is necessary, using Lemma~\ref{lemma:i=0} we can
assume $i_1=0$. Let us note that if the matrix
$\mathbf{M}_J(\beta)$ has rank 2, then all the $3 \times 3$ minors
are zero including
$\det(\mathbf{M}(4,4))=i_2(\beta^{i_3}-1)-i_3(\beta^{i_2}-1),
\det(\mathbf{M}(3,4))=i_2i_3(\beta^{i_3}-\beta^{i_2})$ and $
\det(\mathbf{M}(3,3))i_2i_4(\beta^{i_4}-\beta^{i_2})$, where
$\mathbf{M}(i,j)$ is the submatrix obtained by removing the row
$i$ and column $j$ from $M_J(\beta)$. From the last two,
$\beta^{i_3}=\beta^{i_2}=\beta^{i_4}$ and using this in the the
first minor we have $\beta^{i_2}=1=\beta^{i_1}$.
\end{proof}

\begin{lemma}\label{lemma:mult_of_0,1} Let $J=\{i_1, i_2, i_3, i_4\}$
be a subset of $ \cA$ such that $i_1<i_2<i_3<i_4$ and $\beta \in
K$ a root of unity. Consider the powers $\beta^{i_1}, \dots,
\beta^{i_4}$ of $\beta$. If all four are equal (e.g. $\beta=1$),
then $\beta$ is a root of $D_J$ of multiplicity $4$. If exactly
three of the powers are equal and the remaining one is different,
then the multiplicity of $\beta$ is 1.
\end{lemma}
\begin{proof}
For the multiplicity of $\beta$ when all the powers of $\beta$ are
equal, assuming without loss of generality that $i_1>3$, the first
4 derivatives of $D_J$ can be calculated to see that
$D_J(\beta)=D'_{J}(\beta)=D^{(2)}_J(\beta) =D^{(3)}_J(\beta)=0$
while $D^{(4)}_J(\beta)={2(i_1 - i_2)(i_1 - i_3)(i_1 - i_4)(i_2 - i_3)
  (i_2 - i_4)(i_3 - i_4)\beta^{2i_1 - 4}} \neq 0.$

When only three of the powers of $\beta$ are equal, we can assume
$\beta^{i_1}\neq \beta^{i_2}=\beta^{i_3}=\beta^{i_4}$. Then it is
clear that $D_J(\beta)=0$, but $D'_J(\beta)=
\beta^{i_2-1}(\beta^{i_1} -
\beta^{i_2})(i_4-i_3)(i_4-i_2)(i_3-i_2)\neq 0.$
\end{proof}

\begin{lemma}\label{lemma:valD_I} Let $J=\{i_1,i_2,i_3,
i_4\}\subset \cA$ such that $g=\emph{gcd}(i_3-i_1,i_2-i_1)>1$ and
$\#J=4$. Let $\beta \in K\setminus\{1\}$ be a root of unity such that
$\beta^g=1$. Consider $b=\beta+ht^v+*\in \K$, where $h \in
K\setminus\{0\}$ and $v>0$. Then, the valuation of $D_J(b)$ is $4v$ if
$\beta^{i_4-i_1}=1$, and $v$ otherwise.
\end{lemma}
\begin{proof}To compute $D_{J}(b)$ (up
to a sign) we can assume $i_1<i_2<i_3$ and use the $i_4$-th column
of $\mathbf{M}$ as the last one of $D_J$. By Lemma~\ref{lemma:i=0}
and the facts that $b\ne 0$ and $\val(b^{2i_1})=0$ we can assume
$i_1=0$. Also, by the formulation of $D_J$ in the same Lemma
(where now $i_4$ is possibly a negative number), {\small$$
D_{J}(b) =i_2 i_4 \, b^{i_2}(b^{i_4-i_2}-1)(b^{i_3}-1) - i_2 i_3
\, b^{i_2}(b^{i_3-i_2}-1)(b^{i_4}-1)- i_3 i_4 \,
b^{i_3}(b^{i_4-i_3}-1)(b^{i_2}-1),$$} Let $d>1$ be the order of
$\beta$ as a root of unity. If $d \nmid l$, $b^l-1=\beta^l-1+*$ of
valuation 0, while if $d \mid l$, $b^l= 1+*$ of valuation 0 and
$b^l-1= l\beta^{-1}ht^v+*$ of valuation $v$. If $ d \nmid i_4$,
$D_{J}(b)=-\beta^{-1}h\, i_2\,i_3\,(i_3-i_2) (\beta^{i_4}-1)t^v+*$
which has valuation $v$.

In case that $ d \mid i_4$, let $d':=gcd(i_4,i_3,i_2)$. Note that
$ d \mid d'$. Using the notation $i'_j=\frac{i_j}{d'}$ and
$J'=\{0,i'_2,i'_3,i'_4\}$, we have by Lemma~\ref{lemma:i=0}
$$D_{J}(b)=D_{J'}(b^{d'})=D_{J'}(1+d'\beta^{-1}ht^v+*).$$
By Lemma~\ref{lemma:mult_of_0,1}, the multiplicity of 1 as a root
of $D_{J'}$ is 4, that is there exists a polynomial $p\in \Z[x]$
such that $p(1)\ne 0$ and $D_{J'}(x)=(x-1)^4p(x)$. Then,
$D_{J}(b)=(d'\beta^{-1}ht^v+*)^4p(1+d'\beta^{-1}ht^v+*)$, which
has valuation $4v$.
\end{proof}

\begin{lemma}\label{lemma:alternativa}
Let $\beta\in K\setminus\{0\}$ and $J=\{i_1,i_2,i_3,i_4,i_5\}$ be a subset
of $\cA$ of cardinal 5 such that $D_{J-\{i_5\}}(\beta)=0$ and the
set $\{\beta^{i_1},\beta^{i_2},\beta^{i_3},\beta^{i_4}\}$ does not
have three elements equal and the remaining one different. Then,
either $D_{J-\{i_j\}}(\beta)=0$ for all $ 1 \le j \le 5$, or
$D_{J-\{i_j\}}(\beta)\neq0$ for all $ 1 \le j \le 4$.
\end{lemma}
\begin{proof}
Since we are not assuming that the indices $\{i_j\}_{i=1}^4$ are
ordered, it is enough to prove that if $D_{J-\{i_1\}}(\beta)=0$
then $D_{J-\{i_j\}}(\beta)=0$ for all $ 2 \le j \le 4$. This
result follows from the direct computation below.

Let $y_1, y_2, y_3, y_4, y_5, x_1, x_2, x_3, x_4, x_5$ be
variables over $K$ and $\mathcal{J}_1$ the ideal generated by the
$4\times 4$ minors of the matrix {\small$$M=\begin{pmatrix} 1 & 1
& 1 & 1 & 1 \\ y_1 & y_2 & y_3 & y_4 & y_5
\\ x_1 & x_2 & x_3 & x_4 & x_5 \\ y_1x_1 & y_2x_2 & y_3x_3 &
y_4x_4 & y_5x_5 \end{pmatrix}.$$}Let $f$ be the polynomial $f(y_2,
y_3, y_4)=(y_4-y_3)(y_4-y_2)(y_3-y_2)$ and $\mathcal{I}_2$,
$\mathcal{I}_3$ be the ideals $\mathcal{I}_2=\langle
\det(M_{\{1,2,3,4\}}), \det(M_{\{2,3,4,5\}})\rangle$,
$\mathcal{I}_3=\langle x_{2}-x_3, x_2-x_4\rangle$. Then, by direct
computation (see Computation~\ref{comp:alternativa}) we can check
that the ideals $(\mathcal{I}_2: f)$ and $\mathcal{I}_1\cap
\mathcal{I}_3$ are equal.

If $D_{J-\{i_5\}}(\beta)=D_{J-\{i_1\}}(\beta)=0$, then $(i_1,
\dots, i_5, \beta^{i_1}, \dots, \beta^{i_5})$ is in the variety
defined by $\mathcal{I}_2$ and is not a root of $f$. Then is in
the variety defined by $(\mathcal{I}_2:f)$ and therefore gives us
two choices: \begin{itemize}
\item If the point $(i_1,\dots, i_5, \beta^{i_1}, \dots, \beta^{i_5})$
is in the variety defined by $\mathcal{I}_1$, then
$D_{J-\{i_j\}}(\beta)= 0$ for all $ 1 \le j \le 5$.
\item If the point $(i_1,\dots, i_5, \beta^{i_1}, \dots, \beta^{i_5})$
is in the variety defined by $\mathcal{I}_3$, then
$\beta^{i_2}=\beta^{i_3}=\beta^{i_4}.$ By the hypothesis of the
lemma, these three powers have to be equal to $\beta^{i_1}$ and
hence rank$(M_{J-\{i_5\}}(i_1, \dots, i_5, \beta^{i_1}, \dots,
\beta^{i_5}))=2$. Thus it is also true that
$D_{J-\{i_j\}}(\beta)=0$ for all $ 1 \le j \le 5$.
\end{itemize}
\end{proof}

\begin{lemma} \label{lemma:triple_root}
Let $J=\{i_1,i_2,i_3,i_4\}$ be a subset of $\cA$ and $\beta\in
K\setminus\{0\}$ a root of $D_J$ of multiplicity at least $3$. Then
$\beta$ is a root of unity and $\beta^{i_1}= \beta^{i_2}=
\beta^{i_3}=\beta^{i_4}$.
\end{lemma}
\begin{proof}
Let $y_1, y_2, y_3, y_4, x_1, x_2, x_3, x_4$ be variables over $K$
and $ \mathcal{I}$ the ideal generated by $\det(M), \det(M_1)$ and
$\det(M_{21}) + \det(M_{22})$,
where{\small$$ M=\begin{pmatrix} 1 & 1 & 1 & 1 \\
y_1 & y_2 & y_3 & y_4 \\ x_1 & x_2 & x_3 & x_4 \\
y_1x_1 & y_2x_2 & y_3x_3 & y_4x_4
\end{pmatrix}, M_1=\begin{pmatrix}1 & 1 & 1 & 1 \\
y_1 & y_2 & y_3 & y_4 \\ x_1 & x_2 & x_3 & x_4 \\
y_1^2x_1 & y_2^2x_2 & y_3^2x_3 & y_4^2x_4
\end{pmatrix},$$ $$M_{21}=\begin{pmatrix}1 & 1 & 1 & 1 \\
y_1 & y_2 & y_3 & y_4 \\ x_1 & x_2 & x_3 & x_4 \\
y_1^2(y_1-1)x_1 & y_2^2(y_2-1)x_2 & y_3^2(y_3-1)x_3 &
y_4^2(y_4-1)x_4 \end{pmatrix} \mbox{ and}$$ $\hspace{1,15cm} M_{22}=\begin{pmatrix}1 & 1 & 1 & 1 \\
y_1 & y_2 & y_3 & y_4 \\ y_1x_1 & y_2x_2 & y_3x_3 & y_4x_4 \\
y_1^2x_1 & y_2^2x_2 & y_3^2x_3 & y_4^2x_4
\end{pmatrix}.$}

Note that when we evaluate the matrices in $y_j=i_j$,
$x_j=x^{i_j}$, then $\det(M_1)(x)=\frac{\partial}{\partial
x}\det(M)(x)$ and $\det(M_{21})(x)+\det(M_{22})(x)=
\frac{\partial^2}{\partial x^2}\det(M)(x)$, so $\mathcal{I}$ is
the ideal where $D_J$ and its two derivatives vanish.

Let $f,g$ be the polynomials $f(y_1, y_2, y_3,
y_4)=\prod\limits_{1\le j<k\le 4}(y_k-y_j),$ $
g(x_1,x_2,x_3,x_4)=x_1x_2x_3x_4$ and $\mathcal{I}_1=(\mathcal{I} :
f^\infty)$. Then, by direct computation (See
Computation~\ref{comp:triple_root}) we can check that the ideal
$\mathcal{I}_2=(\sqrt{\mathcal{I}_1}: g)$ is exactly $\langle
x_4-x_3, x_3-x_2,x_2-x_1 \rangle$. If $\beta$ is at least a triple
root of $D_J$, as $(i_1,i_2,i_3,i_4, \beta^{i_1}, \beta^{i_2},
\beta^{i_3}, \beta^{i_4})$ is not a root of $f$ and $\beta\ne 0$,
then is in the variety defined by $\mathcal{I}_2$ and hence
$\beta$ is a root of unity with $\beta^{i_1}= \beta^{i_2}=
\beta^{i_3}=\beta^{i_4}$.
\end{proof}

\begin{lemma} \label{lemma:double_root}
Let $J=\{i_1,i_2,i_3,i_4, i_5\}$ be a subset of $\cA$ of cardinal
5 and $\beta \in K\setminus\{0\}$ a common root of $D_{J-\{i_j\}}$ for all
$ 1 \le j \le 5$ such that the elements in
$\{\beta^{i_j}\}_{j=1}^4$ are repeated at most once. If $\beta$ is a
multiple root of $D_{J-\{i_5\}}$, then it is a multiple root of
$D_{J-\{i_j\}}$ for all $ 1 \le j \le 4$.
\end{lemma}
\begin{proof} Without loss of generality we can assume
$\beta^{i_1}\neq \beta^{i_2}$ and $\beta^{i_1}\neq \beta^{i_3}$.

Let $y_1, y_2, y_3, y_4,y_5, x_1, x_2, x_3, x_4,x_5$ be variables
over $K$ and $\mathcal{I}$ the ideal generated by the $4 \times 4$
minors of the matrix $M$ and $\det((M_1)_{\{1,2,3,4\}})$ where
{\small$$ M=\begin{pmatrix} 1 & 1 & 1 & 1 & 1 \\
y_1 & y_2 & y_3 & y_4 & y_5 \\ x_1 & x_2 & x_3 & x_4 & x_5 \\
y_1x_1 & y_2x_2 & y_3x_3 & y_4x_4 & y_5x_5
\end{pmatrix}, M_1=\begin{pmatrix}1 & 1 & 1 & 1 & 1 \\
y_1 & y_2 & y_3 & y_4 & y_5 \\ x_1 & x_2 & x_3 & x_4 & x_5 \\
y_1^2x_1 & y_2^2x_2 & y_3^2x_3 & y_4^2x_4 & y_5^2x_5
\end{pmatrix}.$$}
Let $f$ be the polynomial $f(y_1, y_2, y_3,
y_4)=\prod\limits_{1\le j<k\le 4}(y_k-y_j)$,
$\mathcal{I}_1=\langle x_1-x_2, x_1-x_3\rangle$ and
$\mathcal{I}_2=(\mathcal{I} : \mathcal{I}_1)$. Then, by direct
computation (see Computation~\ref{comp:double_root}) we can check
that the $4 \times 4$ minors of the matrix $M_1$ are elements of
the ideal $\mathcal{I}_3= (\mathcal{I}_2: f^\infty) $. If $\beta$
is a common root of $D_{J-\{i_j\}}$ for all $ 1 \le j \le 4$ and a
multiple root of $D_{J-\{i_5\}}$, then $(i_1,i_2,i_3,i_4,i_5,
\beta^{i_1},\beta^{i_2},\beta^{i_3},\beta^{i_4}, \beta^{i_5})$ is
in the variety defined by $\mathcal{I}_3$ and hence $\beta$ is a
multiple root of $D_{J-\{i_j\}}$ for all $ 1 \le j \le 4$.
\end{proof}

\begin{lemma}\label{lemma:goodsecondmonomial}
Let $J=\{i_1,i_2,i_3,i_4\}$ be a subset of $\cA$ and $\beta \in
K\setminus\{0,1\}$ a root of $D_J$ such that the powers in
{$\{\beta^{i_j}\}_{j=1}^4$ are repeated at most once. Then there}
exists $i_5 \in \cA-{J}$ such that $ D_{(J\cup\{i_5\})-\{i_j\}} $
does not vanish at $\beta$ for all $ 1 \le j \le 4$.
\end{lemma}
\begin{proof}
Note that if $\beta$ is a root of $D_{(J\cup\{i_5\})-\{i_1\}}$,
then by Lemma~\ref{lemma:alternativa} $\beta$ it is a root of
every $D_{(J\cup\{i_5\})-\{i_j\}}$ for all $ 1 \le j \le 4$. We
now see this can not happen for every $ i_5 \in
\{i_1,i_1+1,i_1+2,i_1+3\}$.

Let $y_1, y_2, y_3, y_4, x_1, x_2, x_3, x_4, x$ be variables over
$K$ and $\mathcal{I}$ the ideal generated by the $4 \times 4$
minors of the matrices $M_i$ for all $ 0 \le i \le 3$
where{\small$$ M_i=\begin{pmatrix} 1 & 1 & 1 & 1 & 1 \\
y_1 & y_2 & y_3 & y_4 & y_1+i \\ x_1 & x_2 & x_3 & x_4 & x_1x^i \\
y_1x_1 & y_2x_2 & y_3x_3 & y_4x_4 & (y_1+i)x_1x^i
\end{pmatrix}.$$}
Let $f$ be the polynomial $f(y_1, y_2, y_3,
y_4)=\prod\limits_{1\le j<k\le 4}(y_k-y_j)$. By direct computation
(see Computation~\ref{comp:goodsecondmonomial}) we can check the
equality of ideals $(\mathcal{I} : x_1x_4x(x-1)^4f)=\langle
x_1-x_4, x_2-x_4, x_3-x_4\rangle$. If
$D_{(J\cup\{i_1+i\})-\{i_j\}}(\beta)=0$ for all $1 \le j \le 4$
and $ 0 \le i \le 3$, then $(i_1, \dots, i_4, \beta^{i_1}, \dots,
\beta^{i_4}, \beta)$ is in the variety defined by $\mathcal{I}$
but is not a zero of $x_1x_4x(x-1)^4f(y_1, \dots, y_4)$. Hence
$\beta^{i_1}=\beta^{i_2}=\beta^{i_3}=\beta^{i_4}$ which
contradicts the statement. Then, there exists $i \in \{0,1,2,3\}$
such that $i_5=i_1+i$ fulfills $
D_{(J\cup\{i_5\})-\{i_j\}}(\beta)\ne 0 $ for all $ 1 \le j \le 4$.
\end{proof}

We present now two technical results which are needed in Section
\ref{sec:excepConfig}:

\begin{lemma}\label{lemma:diofantic_system}
Let $(n, m) \in \N^2$ with $n > 1$ such that $n$ is multiple of
$m$ and $m+1$ is multiple of $n-1$. Then $(n,m) \in
\{(2,1),(3,1),(2,2),(4,2),(3,3)\}$.
\end{lemma}
\begin{proof}
As $n$ is multiple of $m$ and $m+1$ is multiple of $n-1$, then
$m\leq n\leq m+2$. For $m \ge 3$, this inequality implies $n=m$.
{Thus, if} $m\geq 4$, then $(m+1)/(m-1)<2$ and this leads to a
contradiction because $m+1$ is a multiple of $m-1$. Analyzing the
possible cases for $m=1,2$ we obtain the ordered pairs from the
statement.
\end{proof}

\begin{theorem} \label{teo:excepConfig} Let $J=\{i_1,i_2,i_3,i_4\}$
be a subset of $\cA$. For every root $\beta \in K\setminus\{0,1\}$ of
$D_J(x)$ there are at least three powers $\{\beta^{i_1},
\beta^{i_2}, \beta^{i_3}, \beta^{i_4}\}$ equal, if and only if
$J$ is the affine image $(J'(\times s))(+r)$ of
an exceptional configuration $J'$.
\end{theorem}
\begin{proof}
Consider $\beta\neq 0,1$. The set
$\{\beta^{i_1},\beta^{i_2},\beta^{i_3},\beta^{i_4}\}$ has three
elements equal if $\beta$ is a $s_j$-th root of unity for some $j
\in \{1,2,3,4\}$, where $s_1=\gcd(i_3-i_2,i_4-i_2)$,
$s_2=\gcd(i_4-i_1,i_3-i_1)$, $s_3=\gcd(i_4-i_1,i_2-i_1)$,
$s_4=\gcd(i_3-i_1,i_2-i_1)$. Note that, if all the powers of
$\beta$ are equal, $\beta$ is a $s$-th root of unity where
$s=\gcd({s_1,s_2,s_3,s_4})=\gcd(i_4-i_1,i_3-i_1,i_2-i_1)$. Moreover,
note that if $\beta$ is a $s_j$-th root of unity for $j \in
\{1,2,3,4\}$, then either $\beta^{i_4-i_3}=1$ or
$\beta^{i_2-i_1}=1$, and both equalities hold if and only if
$\beta^s=1$.

By Lemma~\ref{lemma:mult_of_0,1}, since every $\beta\ne 0,1$ root
of $D_J$ satisfy $\beta^{s_j}=1$ for some $j \in \{1,2,3,4\}$,
$\beta$ has multiplicity 4 if $\beta^s=1$ and multiplicity 1
otherwise. Then, the set of roots of $D_J$ (counted with
multiplicity) that make at least three $\beta^{i_j}$ equal are at
most $4s+(i_4-i_3-s)+(i_2-i_1-s)=i_4-i_3+i_2-i_1+2s$. Also by
{Lemma~\ref{lemma:B3}}, $D_J$ has $i_4+i_3-i_2-i_1$ nonzero
roots. As every root of $D_J$ makes three powers of $\beta$ equal,
$$i_4+i_3-i_2-i_1\le i_4-i_3+i_2-i_1+2s,$$ or equivalently $i_3 -i_2
\le s$. But as $s$ divides $i_3-i_2$, $s=i_3-i_2$ and therefore
$s_1=s_4=s$. Thus, the only way to have three powers equal and one
different is if $\beta$ is a $s_2$-th or $s_3$-th root of unity,
and therefore
$i_4+i_3-i_2-i_1 \le 4s + (s_2-s)+ (s_3-s)$.
As $s_2\leq i_4-i_3$, $s_3\leq i_2-i_1$ and $2s= 2i_3-2i_2$, so
$$4s+(s_2-s)+(s_3-s)\le i_4-i_3+i_2-i_1+2s=i_4+i_3-i_2-i_1.$$
Hence, all inequalities are, in fact, equalities.

Finally, we analyze the cases when this is possible, that is when
$s_2=i_4-i_3$ and $s_3=i_2-i_1$. As $s_2=\gcd(i_4-i_3,i_3-i_1)$
and $s_3=\gcd(i_4-i_2,i_2-i_1)$, this happens if and only if
$i_3-i_1$ is a multiple of $i_4-i_3$ and $i_4-i_2$ is a multiple
of $i_2-i_1$. Considering the integers $n=\frac{i_3-i_1}{s}$,
$m=\frac{i_4-i_3}{s}$, then $n$ is a multiple of $m$ and (as
$i_3-i_2=s$) $m+1=\frac{i_4-i_2}{s}$ is a multiple of
$n-1=\frac{i_2-i_1}{s}$. By Lemma~\ref{lemma:diofantic_system},
the possible values for $(n, m)$ are:
\begin{itemize}
\item $(n,m)=(2,1)$ in  which case $(i_1, i_2, i_3, i_4)=
((0,1,2,3)(\times s))(+i_1)$.
\item $(n,m)=(3,1)$ in  which case $(i_1, i_2, i_3, i_4)=
((0,2,3,4)(\times s))(+i_1)$.
\item $(n,m)=(2,2)$ in  which case $(i_1, i_2, i_3, i_4)=
((0,1,2,4)(\times s))(+i_1)$.
\item $(n,m)=(4,2)$ in  which case $(i_1, i_2, i_3, i_4)=
((0,3,4,6)(\times s))(+i_1)$.
\item $(n,m)=(3,3)$ in  which case $(i_1, i_2, i_3, i_4)=
((0,2,3,6)(\times s))(+i_1)$.
\end{itemize}
As these are exactly the exceptional configurations under affine
maps, this completes the proof.
\end{proof}

\section{Computations}\label{app:SAGE}

We presented several lemmas that require symbolic
computations which can be done using mathematical software. We have
implemented these computations in Sage~\cite{SA} and we reproduce here the codes
that we have used to check that the claims are correct.

\begin{comp}[For Proposition~\ref{prop:exceptionalConfig}]\label{comp:exceptionalConfig}
\mbox{}
\begin{verbatim}
sage: K.<c6,beta,delta,epsilon,x> = QQ[]
sage: f = c6*(x-1)**2*(x-beta)**2*(x^2+delta*x+epsilon)
sage: c1 = f.coefficient({x:1})
sage: c2 = f.coefficient({x:2})
sage: c3 = f.coefficient({x:3})
sage: c4 = f.coefficient({x:4})
sage: c5 = f.coefficient({x:5})
sage: I = Ideal(c1,c4,c5)
sage: beta*c2*c3 in I
True
\end{verbatim}
\end{comp}

\begin{comp}[For Lemma~\ref{lemma:alternativa}]\label{comp:alternativa}
\mbox{}
\begin{verbatim}
sage: K.<y1,y2,y3,y4,y5,x1,x2,x3,x4,x5>=QQ[]
sage: M=matrix(4,5,[1,1,1,1,1,y1,y2,y3,y4,y5,x1,x2,x3,x4,x5,\
      y1*x1,y2*x2,y3*x3,y4*x4,y5*x5])
sage: N = M.minors(4)
sage: I1 = Ideal(N); I2 = Ideal(N[0],N[-1]); I3 = Ideal(x2-x3,x2-x4)
sage: I2.quotient(Ideal((y4-y3)*(y4-y2)*(y3-y2))) == \
      I1.intersection(I3)
True
\end{verbatim}
\end{comp}

\begin{comp}[For Lemma~\ref{lemma:triple_root}]\label{comp:triple_root}
\mbox{}
\begin{verbatim}
sage: K.<y1,y2,y3,y4,x1,x2,x3,x4>=QQ[]
sage: M=matrix([[1,1,1,1],[y1,y2,y3,y4],[x1,x2,x3,x4],\
      [y1*x1,y2*x2,y3*x3,y4*x4]])
sage: M1=matrix([[1,1,1,1],[y1,y2,y3,y4],[x1,x2,x3,x4],\
      [y1*y1*x1,y2*y2*x2,y3*y3*x3,y4*y4*x4]])
sage: M21=matrix([[1,1,1,1],[y1,y2,y3,y4],[x1,x2,x3,x4],\
      [y1*y1*(y1-1)*x1,y2*y2*(y2-1)*x2,y3*y3*(y3-1)*x3,\
      y4*y4*(y4-1)*x4]])
sage: M22=matrix([[1,1,1,1],[y1,y2,y3,y4],[y1*x1,y2*x2,y3*x3,\
      y4*x4],[y1*y1*x1,y2*y2*x2,y3*y3*x3,y4*y4*x4]])
sage: I=Ideal(M.det(),M1.det(),M21.det()+M22.det())
sage: vandermonde=matrix([[1,1,1,1],[y1,y2,y3,y4],
      [y1**2,y2**2,y3**2,y4**2],[y1**3,y2**3,y3**3,y4**3]]).det()*K
sage: I1=I.quotient(vandermonde)
sage: while I != I1:
             I=I1
             I1=I1.quotient(vandermonde)
sage: I.radical().quotient(Ideal(x1*x2*x3*x4))
Ideal (x4 - x3, x4 - x2, x4 - x1) of Multivariate Polynomial Ring
in y1, y2, y3, y4, x1, x2, x3, x4 over Rational Field
\end{verbatim}
\end{comp}

\begin{comp}[For Lemma~\ref{lemma:double_root}]\label{comp:double_root}
\mbox{}
\begin{verbatim}
sage: K.<y1,y2,y3,y4,y5,x1,x2,x3,x4,x5>=QQ[]
sage: M=matrix([[1,1,1,1,1],[y1,y2,y3,y4,y5],[x1,x2,x3,x4,x5],
      [y1*x1,y2*x2,y3*x3,y4*x4,y5*x5]])
sage: M1=matrix([[1,1,1,1],[y1,y2,y3,y4],[x1,x2,x3,x4],
      [y1*y1*x1,y2*y2*x2,y3*y3*x3,y4*y4*x4]])
sage: I=Ideal(M.minors(4))+Ideal(M1.det())
sage: vandermonde=matrix([[1,1,1,1],[y1,y2,y3,y4],
      [y1**2,y2**2,y3**2,y4**2],[y1**3,y2**3,y3**3, y4**3]]).det()*K
sage: I=I.quotient(Ideal(x1-x2,x1-x3))
sage: I2=I.quotient(vandermonde)
sage: while I != I2:
             I=I2
             I2=I2.quotient(vandermonde)
sage: M1(y1=y5,x1=x5).det() in I
True
sage: M1(y2=y5,x2=x5).det() in I
True
sage: M1(y3=y5,x3=x5).det() in I
True
sage: M1(y4=y5,x4=x5).det() in I
True
\end{verbatim}
\end{comp}

\begin{comp}[For Lemma~\ref{lemma:goodsecondmonomial}]\label{comp:goodsecondmonomial}
\mbox{}
\begin{verbatim}
sage: K.<y1,y2,y3,y4,y5,x1,x2,x3,x4,x5,x>=QQ[]
sage: M0=matrix([[1,1,1,1],[y1,y2,y3,y4],[x1,x2,x3,x4],
      [y1*x1,y2*x2,y3*x3,y4*x4]])
sage: M=matrix([[1,1,1,1,1],[y1,y2,y3,y4,y5],[x1,x2,x3,x4,x5],
      [y1*x1,y2*x2,y3*x3,y4*x4,y5*x5]])
sage: I=Ideal(M0.det())
sage: I=I+Ideal(M(y5=y1,x5=x1).minors(4))
sage: I=I+Ideal(M(y5=y1+1,x5=x1*x).minors(4))
sage: I=I+Ideal(M(y5=y1+2,x5=x1*x**2).minors(4))
sage: I=I+Ideal(M(y5=y1+3,x5=x1*x**3).minors(4))
sage: I.quotient(Ideal((y1-y2)*(y1-y3)*(y1-y4)*(y2-y3)*(y2-y4)
      *(y3-y4)*x1*x4*x*(x-1)**4))
Ideal (x3 - x4, x2 - x4, x1 - x4) of Multivariate Polynomial Ring
in y1, y2, y3, y4, y5, x1, x2, x3, x4, x5, x over Rational Field
\end{verbatim}
\end{comp}

\section*{Acknowledgments}

\noindent AD is supported by UBACYT  20020100100242, CONICET PIP
11220110100580 and ANPCyT 2013-1110, Argentina. MIH is supported
by UBACyT 20020120100133, CONICET PIP 0099/11 and ANPCyT
2013-1110. She is grateful to the Simons Institute for the Theory
of Computing, Berkeley, USA, where this work was developed. LFT {is
supported by the Spanish Ministerio de Econom\'{\i}a y
Competitividad and by the European Regional Development Fund (ERDF),
under the project  MTM2014-54141-P.} We would like to
thank Alexander Esterov for useful comments. We are also grateful to Florian Block
for his enthusiasm to start this project and for his
participation at the early stages of this work.


\end{document}